\documentclass[10pt,reqno]{amsart}
\usepackage[margin=1.25in]{geometry}
\usepackage{amsmath, amssymb, amsfonts, bm, mathrsfs, dsfont}
\usepackage{graphicx,color}
\usepackage{thm-restate, hyperref, cleveref, enumitem}

\makeatletter
\def\namedlabel#1#2{\begingroup
   \def\@currentlabel{#2}%
   \label{#1}\endgroup
}
\makeatother

\theoremstyle{plain}
\newtheorem{theorem}{Theorem}[section]
\newtheorem{proposition}[theorem]{Proposition}
\newtheorem{lemma}[theorem]{Lemma}
\newtheorem{corollary}[theorem]{Corollary}

\theoremstyle{definition}
\newtheorem{definition}[theorem]{Definition}
\newtheorem{example}[theorem]{Example}

\theoremstyle{remark}
\newtheorem{remark}[theorem]{Remark}

\numberwithin{equation}{section}

\newcommand{\R}{{\mathbb R}}
\newcommand{\Z}{{\mathbb Z}}
\newcommand{\N}{{\mathbb{N}}}

\newcommand{\scrJ}{{\mathscr{J}}}
\newcommand{\scrM}{{\mathscr{M}}}
\newcommand{\scrL}{{\mathscr{L}}}
\newcommand{\scrH}{{\mathscr{H}}}
\newcommand{\scrO}{{\mathscr{O}}}
\newcommand{\cA}{{\mathcal{A}}}
\newcommand{\cJ}{{\mathcal{J}}}

\newcommand{\id}{{\mathrm{id}}}
\newcommand{\crit}{{\mathrm{Crit}}}
\newcommand{\critval}{{\mathrm{CritVal}}}
\newcommand{\Ker}{{\operatorname{Ker}}}
\newcommand{\dist}{{\operatorname{dist}}}
\newcommand{\ev}{{\operatorname{ev}}}

\renewcommand{\baselinestretch}{1.5}
\begin{document}
\title[RFH for Tentacular Hamiltonians]{Rabinowitz Floer Homology for Tentacular Hamiltonians}


\author{F. Pasquotto}
\address{Department of Mathematics, Vrije Universiteit Amsterdam, the Netherlands}
\curraddr{}
\email{f.pasquotto@vu.nl}
\thanks{}

\author{R. Vandervorst}
\address{Department of Mathematics, Vrije Universiteit Amsterdam, the Netherlands.}
\curraddr{}
\email{r.c.a.m.vander.vorst@vu.nl}
\thanks{}

\author{J. Wi\'{s}niewska}
\address{Department of Mathematics, Eidgen\"{o}ssische Technische Hochschule Z\"{u}rich, Switzerland.}
\curraddr{}
\email{jagna.wisniewska@math.ethz.ch}
\thanks{}

\subjclass[2010]{53D40; 57R17; 37J50; 70H05}

\keywords{Rabinowitz Floer homology, Morse-Bott action functional, non-compact hypersurfaces, closed characteristics}

\date{\today}

\dedicatory{}

\begin{abstract}
This paper extends the definition of Rabinowitz Floer homology to non-compact 
hypersurfaces. 
We present a general framework for the construction of Rabinowitz Floer homology in the non-compact setting under suitable compactness assumptions on the periodic orbits and the moduli spaces of Floer trajectories. We introduce a class of hypersurfaces arising as the level sets of specific Hamiltonians: strongly tentacular Hamiltonians for which the compactness conditions are satisfied, cf. \cite{pasquotto2017}, thus enabling us to define the Rabinowitz Floer homology for this class.
Rabinowitz Floer homology in turn serves as a tool to address the Weinstein conjecture and establish existence of closed characteristics for non-compact contact manifolds.
\end{abstract}
\maketitle
\section{Introduction}
Rabinowitz Floer homology was defined in \cite{CieliebakFrauenfelder2009} as an invariant of exact contact hypersurfaces in exact convex symplectic manifolds. 
This paper represents the first step towards extending the definition of Rabinowitz Floer homology to include non-compact hypersurfaces.
It is the homology of a complex generated by the critical points of the Rabinowitz action functional and, by construction, it reduces to the singular homology of the hypersurface if the latter does not carry any closed characteristic. 
Because of this property, it is a suitable tool for studying the question of existence of closed characteristics on contact type hypersurfaces, a question known as the (non-compact) Weinstein Conjecture. 
The  Weinstein conjecture for compact contact manifolds has a rich history in symplectic topology: it has been the driving force behind many important developments in symplectic topology, most notably the study of pseudo-holomorphic curves in symplectizations of contact manifolds.
The majority of the literature on the Weinstein Conjecture is devoted to the case of compact hypersurfaces. 
The approach in this paper applies to both compact and non-compact hypersurfaces.
In the case of non-compact hypersurfaces the problem is much harder and additional topological and geometrical conditions are needed as the partial results in \cite{BPV1} suggest. 
The goal of the current paper is to show that Rabinowitz Floer homology can be defined for a class of non-compact hypersurfaces in the standard symplectic space $(\mathbb{R}^{2n},\omega_0)$. 
Some of the results in this  work are presented in a more general setting so as to serve as a first step towards a more general Floer theoretic treatment of the non-compact Weinstein conjecture.
The present  work builds on  previous results concerning $L^{\infty}$-bounds for moduli spaces of Floer trajectories corresponding to a class of Hamiltonians with non-compact regular level sets, cf.\ \cite{pasquotto2017}. In \cite{pasquotto2017} such Hamiltonians are referred to as \emph{tentacular Hamiltonians}: in fact, in order for the Rabinowitz Floer homology to be well-defined, we need to restrict to Hamiltonians in a subclass, which will be referred to as \emph{strongly tentacular Hamiltonians} and which will be defined below. 

A continuous function $F$ is called \emph{coercive} if all its sublevel sets $F^{-1}((-\infty,a]),\, a\in\mathbb{R}$ are compact. The Poisson bracket of two functions $F$ and $H$ on a symplectic manifold $(M,\omega)$ is defined as $\{F,H\}:=\omega(X_{F},X_H)$, where $X_{F},X_H$ are Hamiltonian vector fields associated to $H$ and $F$ respectively, i.e. $dH=\omega(\ \cdot\ , X_H)$ and $dF=\omega(\ \cdot\ , X_{F})$. A vector field $X$ on a symplectic manifold $(M,\omega)$ is called a  {Liouville vector field} if $d\iota_X \omega = \omega$. A Liouville vector field $X$ on the standard symplectic space $(\mathbb{R}^{2n},\omega_0)$  is called \emph{asymptotically regular} if  $\Vert DX(x)\Vert \le c$ for some positive constant $c$ and all $x\in\R^{2n}$. The asymptotically regular Liouville vector fields on $\R^{2n}$ are denoted by $\scrL(\R^{2n})$.

\renewcommand{\baselinestretch}{1.2}

\begin{restatable}{definition}{tentH}
\label{tentH}
Let $H\colon \R^{2n} \to \R$ be a smooth Hamiltonian with regular level set $\Sigma=H^{-1}(0)$. Then $H$ is called \emph{strongly tentacular} if the following axioms are satisfied:
\begin{enumerate}
\item[(h1)\namedlabel{h1}{(h1)}] there exist a vector field $X^\dagger\in \scrL(\R^{2n})$ and constants $c,c'>0$,
such that $dH(X^\dagger)(x) \ge c|x|^2 - c'$, for all $x\in \R^{2n}$;
\item[(h2)\namedlabel{h2}{(h2)}] (sub-quadratic growth) $\sup_{x\in \R^{2n}} \Vert D^3 H(x)\Vert \cdot |x| <\infty$;
\item [(h3)\namedlabel{h3}{(h3)}] (restricted contact type) there exist a vector field $X^\ddagger\in \scrL(\R^{2n})$, 
such that $dH(X^\ddagger)(x) >0$, for all $x\in \Sigma$;
\item[(h4)\namedlabel{h4}{(h4)}] there exists a coercive function $F$, defined in a neighborhood of $\Sigma$, such that for all $x\in \Sigma$ either $\{H,F\}(x)\neq 0$ or $\{H,\{H,F\}\}(x)>0,$ as $|x|\to\infty$.
\end{enumerate}
\end{restatable}

\begin{remark}
\label{tentextra}
Hamiltonians for which  Condition \ref{h4} is relaxed by omitting the coercivity condition on $F$ are called
\emph{tentacular}. This class is used in \cite{pasquotto2017} to establish $L^\infty$-bounds.
Hamiltonians which satisfy axioms (h1)-(h3) are called \emph{admissible}.
\end{remark}
\begin{remark}
Strongly tentacular Hamiltonians include a variety of quadratic functions with hyperboloids as their $0$-level sets. Their compact perturbations are also strongly tentacular, provided the contact type property is preserved. 

The Hamiltonian 
$$
H(q_1,q_2,p_1,p_2):=\frac{1}{2}\left(p_1^2+p_2^2+q_1^2-q_2^2-1\right),
$$
on $(\R^4,\omega_0)$ is an example of a strongly tentacular Hamiltonian with non-compact $0$-level set.
These type of examples are also considered in \cite{CieliebakEliashbergPolterovich2017}, where the authors compute the symplectic homology of their sublevel sets.
\end{remark}

One can immediately notice that if we restrict ourselves to asymptotically regular vector fields, then hypothesis (h1), (h2), and (h4) are identical to (H1), (H2), and (H4), respectively, introduced in \cite{Wisniewska2017}. We will show in Section \ref{sec:tent} that the conditions in Definition \ref{tentH} also imply hypothesis (H3) from \cite{Wisniewska2017}.
Conditions (h1)-(h3) in the definition of tentacular Hamiltonians, together with suitable a priori bounds on the set of non-degenerate periodic orbits contained in a fixed action window guarantee the appropriate bounds for the solutions of the Rabinowitz Floer equations, based on the estimates in \cite{pasquotto2017}. In order to guarantee that the action functional is of Morse-Bott type, condition \ref{h4} is employed to ensure a priori bounds on all the non-degenerate periodic orbits (also under perturbations). Rabinowitz Floer homology can be defined generically for tentacular Hamiltonians.
For strongly tentacular Hamiltonians, we show that the definition is stable under suitable perturbations.

The main result of this paper is the definition of Rabinowitz Floer homology for strongly tentacular Hamiltonians:
\begin{restatable}{theorem}{tentRFH}
Rabinowitz Floer homology of strongly tentacular Hamiltonians is well-defined. Moreover, if $\{H_s\}$ is a one-parameter family of tentacular Hamiltonians in the affine space of compactly supported perturbations of a given Hamiltonian $H$, then Rabinowitz Floer homology is constant along $\{H_s\}$.
\label{thm:tentRFH}
\end{restatable}
The set of strongly tentacular Hamiltonians is not invariant under the action of the group of symplectomorphisms, and yet the conditions sufficient to define  Rabinowitz Floer homology for non-compact hypersurfaces are invariant under symplectomorphisms. As a result, we can extend the definition of Rabinowitz Floer homology to the image of the set of the strongly tentacular Hamiltonians under the action of the group of symplectomorphisms:
\begin{restatable}{corollary}{SympTentRFH}
\label{cor:SympTentRFH}
For every strongly tentacular Hamiltonian $H$ and every symplectomorphism $\varphi$, the Rabinowitz Floer homology of $H \circ \varphi$ is well-defined and isomorphic to the Rabinowitz Floer homology of $H$.
\end{restatable}

In general, the Morse homology of a non-compact manifold depends on the choice of a Morse function (cf. \cite{kang}). By choosing, for our construction, a coercive Morse function on the non-compact manifold $\Sigma$, we can ensure not only that the Morse homology is well-defined, but that the invariant we define still has the following property: in the absence of periodic orbits of the Hamiltonian vector field on $\Sigma$, the Rabinowitz Floer homology is isomorphic to the 
Morse homology (for coercive Morse functions) of $\Sigma$, as defined in \cite{schwarz1993}, and hence to the singular homology. This property, combined with the invariance under small compactly supported perturbations, makes it into a useful tool to detect periodic orbits. 

\begin{corollary}
Let $H$ be a strongly tentacular Hamiltonian and suppose that its Rabinowitz Floer homology is not isomorphic to the singular homology of the regular level set $\Sigma=H^{-1}(0)$. Then $\Sigma$ carries periodic orbits of the Hamiltonian vector field $X_H$.
\end{corollary}

In order to actually apply this result to the (non-compact) Weinstein conjecture, one needs to be able to compute Rabinowitz Floer homology: this is a highly non-trivial task, even in simple examples. In the compact case one can 
sometimes argue that the Rabinowitz Floer homology vanishes by relying on the concept of \emph{displaceability}: if a hypersurface is displaceable, then its Rabinowitz Floer homology necessarily vanishes. Unfortunately, in the non-compact case, we do not have the notion of displaceability at our disposal, so we need to be able to carry out explicit computations, involving a careful study of Conley-Zehnder indices and connecting flow lines with cascades.
A powerful tool for computing Floer homologies in various contexts, such as regular Floer homology, symplectic homology, contact homology, etc.\, is to have a suitable continuation principle for the homology theory at hand. In this paper we prove that some special homotopies induce isomorphisms of the associated Rabinowitz Floer homology groups, and we use this result to show that the homology is well-defined, cf.\ Section\ \ref{sec:InvRFH}. We strongly believe that Rabinowitz Floer homology allows a more global continuation principle (i.e., generic homotopies should induce isomorphisms of the homology groups) and that this continuation principle can be used to establish the homology in various cases, such as the tentacular hyperboloids and their perturbations. The extension of the continuation principle will be a subject of future research.

The paper is organized as follows: in Section \ref{setting} we describe our setting and recall some definitions. In Sections \ref{MorseBott} to \ref{sec:InvRFH} we discuss the general framework of the definition of non-compact Rabinowitz Floer homology. More precisely, we consider a pair $(H,J)$ consisting of a Hamiltonian function and a compatible almost complex structure on an arbitrary exact symplectic manifold with the following properties: 
\begin{itemize}
\item the closed characteristics on $H^{-1}(0)$ are bounded, i.e. $H$ satisfies property \ref{PO} as introduced in \cite{pasquotto2017};
\item the Rabinowitz action functional associated to $H$ satisfies the Morse-Bott property;
\item the moduli spaces of Floer trajectories between different components of the critical set are bounded.
\end{itemize}
In this setting we show that the Rabinowitz Floer homology $RFH(H,J)$ is well defined. 
If uniform bounds on the moduli spaces of the perturbed Floer trajectories can be obtained, then the Rabinowitz Floer homology is independent of the choice of almost complex structure and invariant under sufficiently small compactly supported perturbations of $H$.

In the last two sections of the paper we concentrate on strongly tentacular Hamiltonians. 
In Section \ref{sec:tent} we present the proof of Theorem \ref{thm:tentRFH}.
Using the bounds obtained in \cite{pasquotto2017} and given that the Morse-Bott condition is generic in a suitable neighborhood of the affine space of compactly supported perturbations of a given strongly tentacular Hamiltonians we define Rabinowitz Floer homology for every Hamiltonian in this class.

Finally, in Section \ref{sec:SympHyper}, we use H\"{o}rmander's symplectic classification of quadratic forms \cite{Hormander1995} to show that a large number of quadratic Hamiltonians are strongly tentacular and thus the associated Rabinowitz Floer homology is well defined. In analogy with the case of positive definite quadratic forms and symplectic ellipsoids, we also introduce the notion of \emph{symplectic hyperboloid} as a sublevel set of a Hamiltonian coming from a quadratic form with at least one positive and one negative eigenvalue. Our hope is that Rabinowitz Floer homology can eventually be applied to obtain a classification of symplectic hyperboloids, similar to the classification of symplectic ellipsoids obtained as an application of symplectic homology in \cite{FloerHoferWysocki}.

\section{Setting and definitions}\label{setting}

We recall the definition of Rabinowitz Floer homology in the setting of hypersurfaces in an exact symplectic manifold $(M,\omega=d\lambda)$, where $\lambda$ is a primitive of $\omega$ and it is called a Liouville form.
In the last part of the paper we will specialize to the case of the standard symplectic manifold $(\R^{2n},\omega)$.
For a (smooth) Hamiltonian function $H\colon M \to \R$ we define the Hamiltonian vector field $X_H$  by the relation
$\iota_{X_H} \omega = -dH$. Assume without loss of generality that $\Sigma=H^{-1}(0)$ is a regular hypersurface.
Periodic solutions of the Hamilton equations $\dot x  = X_H(x)$ are called closed characteristics.
The objective is to find geometric and topological conditions on the hypersurface $\Sigma$ and/or $H$ that guarantee the existence of closed characteristics. 
Central in this approach is the variational formulation of the problem, where the periodic solutions of Hamilton's equations are critical points of an appropriate action functional.
In the variational formulation one can either fix the period $\eta$, in which case the energy of the Hamiltonian is undetermined, or one can fix the energy, in which case variations in the period are needed.
In order to have an appropriate variational principle, we scale the period such that all periodic solutions are defined
as mappings $\R/\Z \to M$, that is, we define $v(t) = x(\eta t)$. For the Hamilton equations this yields
$ \partial_t v = \eta \dot x = \eta X_H(x) = \eta X_H(v)$. 
The Rabinowitz action functional $\cA^H$ is defined for a pair $(v,\eta) \in C^\infty(\R/\Z;M) \times \R$ by
\[
\cA^H(v,\eta) = \int_0^1 \lambda (\partial_t v) - \eta \int_0^1 H(v).
\]
Notice that the action is independent of the chosen Liouville form.
Critical points of $\cA^H$ with respect to variations in both $v$ and $\eta$ force closed characteristics $v$ to lie on
the hypersurface $\Sigma$ - in other words, the energy $H(v) =0$.
The set of critical points of $\cA^H$ is denoted by $\crit(\cA^H)$.
Other energy values $e$  can be chosen by replacing the Hamiltonian $H$ by $H-e$.
We allow $\eta$ to be positive or negative. If $(v,\eta)$ is a critical point of $\cA^H$, then $(\bar{v},-\eta)$, where $\bar{v}(t):=v(-t)$, is also a critical point. The pairs $(\bar{v},-\eta)$ and $(v,\eta)$ have opposite orientation.
Pairs of the form $(v,0)$ correspond to constant loops at points of $\Sigma$.
The Hamilton equations for a non-trivial closed characteristic 
are $\partial_t v = \eta X_H(v)$, $\eta\not = 0$ 
and $H(v(t))=0$.
Via the reparametrisation $v(t) := x(\sigma)$ we have $\dot{x} = X_H(x)$ where $\sigma =\eta t$, $\dot{x}$ the derivative with respect to $\sigma$ and $\eta$ is the period.
The Rabinowitz action for pairs $u=(v,\eta)$ records every periodic orbit infinitely many times, i.e.
given $x$ a non-trivial closed characteristic with period $\tau>0$, then
\[
u_k = (v_k,\eta_k),\quad v_k(t) = x(\eta_k t), \quad \eta_k = k\tau, k\in \Z,
\]
are critical points on $\cA^H$ which correspond to traversing the given closed characteristics $k$ times.
If $k$ is negative then an orbit is traversed in the opposite direction (opposite orientation) and the action
has the opposite sign.

We study critical points of $\cA^H$ by considering the appropriate positive gradient flow equation:
$\partial_s u = \nabla \cA^H(u)$, where $\nabla$ is the gradient with respect to an auxiliary Riemannian structure and $u=(v,\eta)\in C^\infty(\R\times(\R/\Z),M)\times C^\infty(\R)$. In order to do so we now define natural Riemannian structures.
An almost complex structure compatible with $\omega$ is a smooth field of complex structures on the tangent bundle $TM$, i.e. for every $x\in M$, $J(x)\colon T_x M \to T_x M$ is a linear mapping such that $J(x)^{2}=-\id$ and the assignment $g_J(\cdot ,\cdot) := \omega(\cdot, J(x) \cdot)$
defines a Riemannian metric on $M$. Let $\cJ(M,\omega)$ be the set of $\omega$-compatible, almost complex structures on $M$. Then, $\cJ(M,\omega)$ is nonempty and contractible as shown in \cite[Prop.\ 13.1]{Silva2001}.
Let us fix $J_0 \in \cJ(M,\omega)$. On $(\R^{2n},\omega_0)$, for instance, we choose $J_0$ to be the standard almost complex structure, i.e. the one for which the associated metric $g_{J_0}$ is the Euclidean metric. From now on all the $C^k$-norms will be taken with respect to the metric $g_{J_0}$ unless stated otherwise. Let
\[
(t,\eta) \mapsto J(\cdot,\eta, t) \in \cJ(M,\omega),
\]
be a smooth function such that outside an open set $V\subseteq M\times \mathbb{R},\ J$ is equal to $J_0$. 
In addition we assume that
\begin{equation}
\sup_{(t,\eta)\in \R/\Z\times \R}\|J(\cdot,\eta,t)\|_{C^{k}}<+\infty, \qquad \forall\ k\in \N,
\label{supJ}
\end{equation}
where the norms are taken with respect to $g_{J_0}$.
The set of such $\omega$-compatible almost complex structures will be denoted by $\scrJ^\infty(M,\omega,V)$.
Note that for every open set $V\subseteq M\times \mathbb{R}$ the corresponding set $\scrJ^\infty(M,\omega,V)$ is also nonempty and contractible. 
For $V=M\times \R$ we use the notation $\scrJ^\infty(M,\omega)$.
Each $J\in\scrJ^\infty(M,\omega)$  induces a metric $g_J$ on $C^{\infty}(\R/\Z;M)\times \R$ defined as follows. For every
$(\xi,\sigma), (\xi',\sigma') \in T_{(v,\eta)}\left(C^{\infty}(\R/\Z;M)\times \R\right) = C^{\infty}(\R/\Z;v^{*}TM)\times \R$ we define the metric $g_J$ by
\begin{equation}
g_J\bigl((\xi,\sigma),(\xi',\sigma')\bigr)=\int_{0}^{1}\omega\Bigl(\xi(t),J\bigl(v(t),\eta,t\bigr) \xi'(t)\Bigr) +\sigma \sigma'.\label{eqn:gJ}
\end{equation}
By assumption \eqref{supJ}, the metrics $g_J$ on $C^{\infty}(\R/\Z;M)\times \R$ induced by $J\in\scrJ^\infty(M,\omega)$ are all equivalent.

The gradient of $\cA^H$ with respect to the metric $g_J$ will be denoted by $\nabla_J \cA^H$ and is given by
\begin{equation}
\label{Jgrad}
\nabla_J \cA^H(v,\eta)=\left( \begin{array}{c}
-J(v,\eta,t)\bigl[\partial_{t}v-\eta X_{H}(v)\bigr] \\
-{\int_0^1 H(v) }\end{array} \right).
\end{equation}
Similarly, we introduce the Hessian of $\cA^H$ as the bilinear functional $\nabla^2_J\cA^H$ defined by
$$
g_{J}\left(\nabla^2_J\cA^H_{(v,\eta)}(\xi,\sigma),(\xi',\sigma')\right):=d_{(v,\eta)}\left(d\cA^H (\xi,\sigma)\right)(\xi',\sigma'),
$$
where $(\xi,\sigma), (\xi',\sigma') \in T_{(v,\eta)}\left(C^{\infty}(\R/\Z;M)\times \R\right)$. The explicit expression for $\nabla^2_J\cA^H$ is:
\begin{equation}
\nabla_J^2\cA^H_{(v,\eta)}(\xi,\sigma)=\left(\begin{array}{c}
-J(v,\eta,t)\bigl[\partial_{t}\xi-\sigma X_{H}(v)-\eta D_vX_H(\xi)\bigr]\\
-\int dH(\xi)
\end{array}
\right). \label{HessAH}
\end{equation}

The gradient flow equations $\partial_s u = \nabla_J \cA^H(u)$ are known as the Rabinowitz Floer equations, and we study solutions of these equations, 
according to the following definition.

\begin{definition}\label{moduli}
Let $J \in \scrJ^\infty(M,\omega)$.
A  {Floer trajectory} is a solution $u=(v,\eta)\in C^{\infty}(\R\times (\R/\Z);M)\times C^\infty(\R,\R)$ of the Rabinowitz Floer equations
\begin{equation}
\label{Jgrad2}
\partial_{s} u = \nabla_J\cA^H(u)=\left( \begin{array}{c}
-J(v,\eta,t)\bigl[\partial_{t}v-\eta X_{H}(v)\bigr] \\
-{\int_0^1 H(v) }\end{array} \right),
\end{equation}
for which 
\begin{equation}
\label{bddcond}
\int_{-\infty}^\infty \Vert \partial_s u\Vert^2_{L^2(\R/\Z)\times\mathbb{R}} ds <\infty.
\end{equation}
For a pair of connected components $\Lambda^-,\Lambda^+$ of the critical set $\cA^H$, the set of Floer trajectories with
$\lim_{s\to \pm \infty}u(s)\in \Lambda^{\pm}$
is called a {moduli space} of Floer trajectories and it is denoted by $\scrM(\Lambda^-,\Lambda^+)$.
\end{definition}

In our study of the Rabinowitz Floer equations, and more precisely when discussing invariance of Rabinowitz Floer homology, it is necessary to consider 1-parameter families of Hamiltonians and almost complex structures.
We denote smooth 1-parameter families of Hamiltonians and almost complex structures by $\{H_s\}_{s\in \R}$ and $\{J_s\}_{s\in \R}$, respectively. Such families are referred to as homotopies.
We will always assume our homotopies $\Gamma=\{(H_{s},J_{s})\}_{s\in\R}$ to be constant outside the interval $[0,1]$, i.e.
 \begin{equation}\label{homotopy1}
(H_{s},J_{s}) = \left\lbrace\begin{array}{l l}
(H_{0}, J_{0}) & s\leq 0\\
(H_{1}, J_{1}) & s \geq 1
\end{array}\right.
\end{equation}
and compactly supported in the Hamiltonian component, i.e., $H_s-H_0\in C_c^\infty(M;\R)$.
Define
\[
\| H_{s}\|_{1,\infty}:=\max_{\substack{x\in M\\ s\in [0,1]}}|\partial_{s}H_{s}(x)|, \qquad
\Vert J_s\Vert_\infty = \max_{\substack{x\in M, t\in \R/\Z\\ \eta\in \R, s\in [0,1]}}\|J_s(x,\eta,t)\|.
\]
The perturbed Rabinowitz Floer equations are now found by replacing $H$ by $H_s$ and $J$ by $J_s$ in Eqn.\ \eqref{Jgrad2}. Floer trajectories for the perturbed Rabinowitz Floer equations 
$\partial_s u = \nabla_{J_s} \cA^{H_s}(u)$ are solutions satisfying the integral bound in \eqref{bddcond}.
The associated moduli spaces will be denoted by $\scrM_\Gamma(\Lambda^-,\Lambda^+)$ with 
$\Lambda^-$ a connected component of  $\crit(\cA^{H_0})$ and $\Lambda^+$ of  $\crit(\cA^{H_1})$.

\section{Morse-Bott property}\label{MorseBott}

Many of the results we prove in this paper, hold in fact under more general, albeit more abstract conditions, which are defined below. Roughly speaking, these conditions require the non-degenerate periodic orbits (all of them or the ones in a prescribed action window) to be confined to a compact set. We will use the expression \emph{uniformly continuous} in order to indicate that these properties persist under compact perturbations. Given a manifold $M$ and a compact subset $K\subset M$, let $C^{\infty}_{c}(K)$ be the space of smooth functions on $M$ with support contained in $K$. We define $C^{\infty}_{c}(M)$ as the union 
\[
C^{\infty}_{c}(M)=\bigcup\limits_{\substack{K\subseteq M,\\ K \textrm{ compact}}} C^{\infty}_{c}(K),
\]
and equip it with the inductive limit topology, that is, we declare a set $U\subset C^{\infty}_{c}(M)$ to be open if and only if $U\cap C^{\infty}_{c}(K)$ is open in $C^{\infty}_{c}(K)$ for every compact subset $K\subset M$.

\begin{restatable}{definition}{variousPO}
\label{variousPO}
Let $H$ be a Hamiltonian on an exact, symplectic manifold $(M,\omega)$, having $0$ as a regular value. The following conditions concern the non-degenerate periodic orbits of $H$ and certain families of perturbations of $H$.
\begin{enumerate}
\item[\ref{PO}\namedlabel{PO}{(PO)}] We say that $H$ satisfies property \ref{PO} if 
for any fixed action window, all the non-degenerate periodic orbits are contained in a compact subset of $M$.
Moreover, we say that property \ref{PO} is \emph{uniformly continuous} at $H$ if there exist an open neighborhood $B(H)$ of $0$ in $C^{\infty}_{c}(M)$ and an exhaustion $\{K_{n}\}_{n\in\mathbb{N}}$ of $M$ by compact sets $K_{n}$, such that for every $n\in \mathbb{N}$ and every $h\in B(K_{n})=B(H)\cap C^{\infty}_{c}(K_{n})$, whenever
$$
(v,\eta)\in {\crit}(\cA^{H+h})\quad\textrm{and}\quad 0<|\cA^{H+h}(v,\eta)|\leq n
$$
then $v(S^{1})\subseteq K_{n}\cap (H+h)^{-1}(0)$. 
\item[\ref{PO+}\namedlabel{PO+}{(PO+)}] We say that $H$ satisfies property \ref{PO+} if all the non-degenerate periodic orbits of the Hamiltonian flow of $H$ are contained in a compact subset. Moreover, we say that property \ref{PO+} is \emph{uniformly continuous} at $H$ if there exist an open neighborhood $B(H)$ of $0$ in $C^{\infty}_{c}(M)$ and an exhaustion $\{K_{n}\}_{n\in\mathbb{N}}$ of $M$ by compact sets $K_{n}$, such that for every $n\in \mathbb{N}$ and every $h\in B(K_{n})=B(H)\cap C^{\infty}_{c}(K_{n})$, whenever
$$
(v,\eta)\in {\crit}(\cA^{H+h})\quad\textrm{and}\quad |\cA^{H+h}(v,\eta)|>0,
$$
then $v(S^{1})\subseteq K_{n}\cap (H+h)^{-1}(0)$. 
\end{enumerate}
\end{restatable}

Condition \ref{PO+} (and its persistence under perturbations) are necessary in this paper for the construction of the homology, but this depends on the specific construction.
In a following paper we will show that \ref{PO} and its uniform continuity are in fact sufficient for the definition of RFH, by constructing this homology as a (direct and inverse) limit of truncated homologies, taken over an increasing, nested family of intervals.

In order to be able to define Rabinowitz Floer for tentacular Hamiltonians, we will at first assume that the following Morse-Bott condition holds for $H$. 
\begin{itemize}
\item[(MB)\namedlabel{MB}{(MB)}] The associated Rabinowitz functional ${\cA^{H}}$ is Morse-Bott and all the periodic orbits are of Morse-Bott type.
\end{itemize}
In fact we will prove below that whenever the \ref{PO+} is uniformly continuous at the Hamiltonian $H$, the Morse-Bott assumption is generically satisfied in the affine space of compactly supported perturbations of $H$, and therefore not a restrictive condition.

We say that a $C^2$ action functional $\cA$ \emph{is Morse-Bott}, if the critical set of $\cA$ consists of a disjoint union of manifolds and for every connected component $\Lambda\subseteq \crit(\cA)$ and every $x\in \Lambda$
$$
T_x\Lambda=\Ker(\nabla^2_x\cA).
$$

On the other hand, the Morse-Bott condition on the periodic orbits has the following meaning:
\begin{definition}
Consider the  projection
$\pi_{{\cA^{H}}} \colon {\crit}({\cA^{H}}) \to \Sigma,\
(v,\eta) \mapsto v(0)$ and let 
$\phi^t$ be the Hamiltonian flow on $\Sigma$.
We say that the closed orbits of the Hamiltonian flow are of \emph{Morse-Bott type} if the following conditions are satisfied:
\begin{enumerate}
\item  $\eta$ is constant on every connected component $\Lambda \subseteq {\crit}({\cA^{H}})$;
\item  the image of $\Lambda$ under the projection
${\mathscr{P}}^{\Lambda}:= \pi_{{\cA^{H}}}(\Lambda)$
is a closed submanifold of $\Sigma$;
\item for all $p\in {\mathscr{P}}^{\Lambda}$
$$
T_{p}{\mathscr{P}}^{\Lambda} = \Ker(D_{p}\phi^{\eta} - Id) \subseteq T_{p} \Sigma.
$$
\end{enumerate}
\label{def:MBT}
\end{definition}
We will first show how these two conditions are related by proving the following lemma:

\begin{lemma}
If ${\cA^{H}}$ is Morse-Bott and $\eta$ is constant on every connected component of ${\crit}({\cA^{H}})$, then all the closed orbits of the Hamiltonian flow $\phi$ on $\Sigma$ are of Morse-Bott type.
\label{lem:MBtype}
\end{lemma}
Observe that the assumption that $\eta$ is constant on every connected component of ${\crit}({\cA^{H}})$ is not very restrictive. In particular, it is satisfied whenever the critical set consists of a union of disjoint circles and the hypersurface $\Sigma\times \{0\}$, or when we assume $H$ to be a defining Hamiltonian 
as defined in \cite{CieliebakFrauenfelder2009}, i.e. when the Reeb vector field coincides with the Hamiltonian vector field: in our setting the first condition is satisfied generically, whereas the authors of \cite{CieliebakFrauenfelder2009} and \cite{Fauck2015} work with defining Hamiltonians.
In the rest of this paper we will denote by $\Sigma_0$ the non-compact component of the critical set of $\cA^{H}$, that is, 
\[
\Sigma_0=\Sigma\times \{0\}\subset C^{\infty}(\mathbb{R}/\mathbb{Z};M)\times \mathbb{R}.
\]
Notice that $\Sigma$ and $\Sigma_0$ are diffeomorphic, but sit inside different ambient manifolds.
\begin{proof}
Fix a connected component $\Lambda \subseteq {\crit}({\cA^{H}})$ and let ${\mathscr{P}}^{\Lambda}$ be the corresponding projection on $\Sigma$. By assumption $\eta$ is constant on $\Lambda$. We want to prove that for all $p\in {\mathscr{P}}^{\Lambda}$
$$
T_{p}{\mathscr{P}}^{\Lambda} = \Ker(D_{p}\phi^{\eta} - Id).
$$
\textbf{Part 1. ($\subseteq $)}\\
Take $\mathtt{v}\in T_{p}{\mathscr{P}}^{\Lambda}$. Then for small enough $\delta>0$ there exists a path $p:(-\delta,\delta)\to \mathscr{P}^{\Lambda},\ p(0)=p$, such that
$$
\mathtt{v}= \frac{d}{ds}\Big|_{s=0}p(s).
$$

By assumption $p(s)\in \mathscr{P}^{\Lambda}$ and $\eta$ is constant on $\Lambda$, hence
for every $s\in (-\delta,\delta)$ the corresponding loop $v_s: \mathbb{R}/ \mathbb{Z} \to M,\ v_s(t):=\phi^{\eta t}(p(s))$ satisfies $(v_s,\eta)\in \Lambda$. In particular, for every $(s,t)\in (-\delta,\delta)\times S^{1},\ 
\phi^{\eta t}(p(s)) \in {\mathscr{P}}^{\Lambda}$.
For every $t\in S^1$ define
$$
\xi(t):= \frac{d}{ds}\Big|_{s=0}\phi^{\eta t}(p(s)) = D_{p}\phi^{\eta t}(\mathtt{v}).
$$
Then $\xi \in v_0^*(T\mathscr{P}^{\Lambda})$ and
$$
D_{p}\phi^{\eta}(\mathtt{v})=\xi(1)= \frac{d}{ds}\Big|_{s=0}\phi^{\eta}(p(s)) = \frac{d}{ds}\Big|_{s=0}p(s) = \xi(0)=\mathtt{v},
$$
which proves the first inclusion.\\
\textbf{Part 2. ($\supseteq $)}\\
Now we want to prove the second inclusion. Take $\mathtt{v}\in \Ker(D_{p}\phi^{\eta} - Id)\subseteq T_{p}\Sigma$ and define
$$
\xi(t):= D_{p}\phi^{\eta t}(\mathtt{v}).
$$
Then
$$
\xi(1)= D_{p}\phi^{\eta}(\mathtt{v}) = \mathtt{v}=\xi(0).
$$
If for $t\in S^1$ we define the loop $v(t):=\phi^{\eta t}(p)$, then by assumption $(v,\eta)\in \Lambda$. In particular, we have that $\xi\in v^*(TM)$. 
Now we would like to calculate $\nabla ^{2}{\cA^{H}}_{(v,\eta)}(\xi,0)$ as defined in \eqref{HessAH}. First observe that since $\mathtt{v}\in T_{p}\Sigma$, we have
$$
\int dH_{v(t)}(\xi)=\int dH_{v(t)}(D_{p}\phi^{\eta t}(\mathtt{v})) = \int dH_{p}(\mathtt{v})=0.
$$
On the other hand, note that we can express the first term of $\nabla ^{2}{\cA^{H}}_{(v,\eta)}(\xi,0)$ using the Lie bracket and Lie derivative as follows:
\begin{align}
\partial_{t}\xi - \eta J_{0} Hess_{v} H(\xi) & = \nabla_{\partial_{t}v}\ \xi - \eta \nabla_{\xi} X_H  = \nabla_{\eta X_H}\xi -\eta \nabla_{\xi}X_H= [\eta X_H,\xi]  = \mathcal{L}_{\eta X_H}\xi \nonumber\\
& = \frac{d}{d s}\Big|_{s=0}D \phi^{-\eta s}\xi(t+s)  = \frac{d}{d s}\Big|_{s=0}D \phi^{-\eta s}D_{p}\phi^{\eta (t+s)}(\mathtt{v}) \nonumber  = \frac{d}{d s}\Big|_{s=0} D_{p}\phi^{\eta t} (\mathtt{v}) \nonumber = 0.
\label{eqn:KerHess}
\end{align}
This proves that 
$$
(\xi,0)\in \Ker(\nabla ^{2}{\cA^{H}}_{(v,\eta)})= T_{(v,\eta)}\Lambda,
$$
in view of the Morse-Bott property of the action functional. Hence $\xi(0)=\mathtt{v}\in T_{p}{\mathscr{P}}^{\Lambda}$, which proves the second inclusion.
\end{proof}

If $H$ is a Hamiltonian on an exact, symplectic manifold $(M,\omega)$, having $0$ as a regular value,
the corresponding Rabinowitz action functional will always be Morse-Bott along the component of the critical set consisting of constant loops, but it may in general not be Morse-Bott along the other components. Nevertheless, one can prove that the Morse-Bott property is generic under certain assumptions. The present proof is inspired by Appendix B of \cite{CieliebakFrauenfelder2009}, which applies to compactly supported Hamiltonians. However, since we consider possibly non-compact energy level sets, we will restrict ourselves to compact perturbations $h\in C^{\infty}_{c}(M)$ and Hamiltonians such that uniform continuity of \ref{PO+} is satisfied. 

\begin{theorem}
Let $(M,\omega)$ be a symplectic manifold and $H:M\to \mathbb{R}$ a smooth Hamiltonian, such that $0$ is a regular value of $H$ and \ref{PO+} is uniformly continuous at $H$. Then there exists an open neighborhood $B(H)$ of $0$ in $C^{\infty}_{c}(M)$, such that the set 
\begin{equation}
A(H):=\{h \in B(H)\ |\ \cA^{H+h}\ \textrm{satisfies}\ \ref{MB} \}
\label{eqn:Hreg}
\end{equation}
 is comeager in $B(H)$. Moreover, if $\Sigma$ is of contact type, then we can assume that for all $h\in A(H)$ the set of critical values of $\cA^{H+h}$ is closed and discrete.
 \label{thm:MB(K)}
\end{theorem}
The first part of the proof is just an adaptation of \cite[Thm. B.1]{CieliebakFrauenfelder2009} to the non-compact setting. In order to prove the second part, though, we will first have to prove a couple of lemmas. The proof of Theorem \ref{thm:MB(K)} can be found at the end of this section.

The first lemma shows how the \ref{PO} and \ref{PO+} conditions guarantee that the value of $\eta$ on the non-degenerate components of ${\crit}({\cA^{H}})$ is (uniformly) bounded away from $0$. This will be a crucial step in proving that the set of critical values of ${\cA^{H}}$ is discrete and, later on, in the construction of a suitable set of perturbations when we prove transversality for tentacular Hamiltonians.
\begin{lemma}
Let $(M,\omega)$ be an exact symplectic manifold, and $H:M\to\mathbb{R}$ a Hamiltonian having $0$ as a regular value. For all $n\in \mathbb{N}$ denote
$$
\eta(H)_{\infty}^{n} := \inf\left\lbrace |\eta|\ \Big|\ (v,\eta)\in {\crit}({\cA^{H}})\cap ({\cA^{H}})^{-1}([-n,0)\cup(0,n])\right\rbrace.
$$
Then the following statements hold:
\begin{enumerate}
\item If $H$ satisfies condition \ref{PO}, then for all $n\in \mathbb{N}$, $\eta(H)_{\infty}^{n}>0$. 
\item If $H$ satisfies condition \ref{PO+}, then $\inf_{n\in\mathbb{N}}\eta(H)_{\infty}^{n}>0$.
\end{enumerate}
\label{lem:etaInf}
\end{lemma}
\begin{proof}
Fix $n\in\mathbb{N}$ and suppose that $\eta(H)_{\infty}^{n}=0$. That means that there exists a sequence 
$$
(v_{k},\eta_{k})\in {\crit}({\cA^{H}})\cap ({\cA^{H}})^{-1}([-n,0)\cup(0,n])\},
$$
such that $\lim_{k\to \infty}\eta_{k}=0$.

On the other hand, by property (PO) there exists a compact set $K\subseteq \Sigma$, such that for all $k\in \mathbb{N}$, we have $v_{k}(S^{1}) \subseteq K$ and hence
$$
\|\partial_{t}v_{k}(t)\|_{L^{\infty}(S^{1})} \leq |\eta_{k}| \sup_{K}|X_H|.
$$
Therefore, by the Arzel\'{a}-Ascoli theorem, there exists a convergent subsequence, which we will denote in the same way:
$$
\lim_{k\to\infty}(v_{k},\eta_{k})=(v,\eta).
$$
By continuity, $\eta=0$ and $\partial_{t}v=0$. Moreover, by continuity of $\nabla{\cA^{H}}$, $(v,\eta)$ is a critical point of ${\cA^{H}}$. Therefore $(v,\eta)\in \Sigma_0$. Yet, this would contradict the fact that $\Sigma_0$ is isolated in ${\crit}({\cA^{H}})$ due to the fact that for every Hamiltonian the corresponding action functional is Morse-Bott along $\Sigma_0$, as proven in Theorem B.1, Step 4 in \cite{CieliebakFrauenfelder2009}.
Indeed, by the arguments presented in the proof of \cite[Thm. 23]{Fauck2016} and by the Morse-Bott property for every $(v,0)\in\Sigma_0$ we can find a neighborhood of $(v,0)$, which does not contain any critical points of ${\cA^{H}}$ except those in $\Sigma_0$. That gives a contradiction and proves the first result.

To prove the second result we argue analogously just for a sequence $(v_{k},\eta_{k})\in {\crit}({\cA^{H}})\setminus ({\cA^{H}})^{-1}(0)$ with $\lim_{k\to \infty}\eta_{k}=0$.
\end{proof}

The previous lemma shows that, under the \ref{PO} assumption, $\eta$ is bounded away from $0$ on the nondegenerate components of  ${\crit}({\cA^{H}})$. Moreover, we will show below that ${\critval}({\cA^{H}})$ is closed and discrete.
\begin{lemma}
\label{lem:critVal}
Let $(M,\omega=d\lambda)$ be an exact symplectic manifold and $H :M\to \mathbb{R}$ a smooth Hamiltonian, such that
\begin{enumerate}
\item $\Sigma=H^{-1}(0)$ is a smooth hypersurface of contact type;
\item the corresponding action functional ${\cA^{H}}$ is Morse-Bott;
\item for any fixed action window, the image of the non-degenerate periodic orbits in this action window is contained in a compact subset of $\Sigma$ (property \ref{PO}).
\end{enumerate}
Then its set of critical values ${\critval}({\cA^{H}})$ is closed and discrete.
\end{lemma}
\begin{proof}
First we will prove that the set ${\critval}({\cA^{H}})$ is closed. Take a sequence $\{a_{k}\}\subset {\critval}({\cA^{H}})$, such that $\lim_{n\to\infty}a_{k}=a$. Without loss of generality we can assume that $a_{k}\neq 0$. Therefore, there exist 
$$
(v_{k},\eta_{k})\in {\crit}({\cA^{H}})\cap \left({\cA^{H}}\right)^{-1}([\inf_{k\in\mathbb{N}}a_{k},\sup_{k\in\mathbb{N}}a_{k} ] \setminus\{0\}), \qquad \textrm{such that} \qquad {\cA^{H}}(v_{k},\eta_{k})=a_{k}.
$$
In particular, by property \ref{PO}, there exists a compact subset $K\subseteq M$, such that for all $k\in \mathbb{N}$ we have 
$v_{k}(S^{1}) \subseteq K$. Since $\Sigma$ is of contact type, there exists a Liouville vector field $Y$ for which
$$
a_{n} = {\cA^{H}}(v_{n},\eta_{n})  = \eta_{n} \int dH_{v_{n}}(Y),\qquad\qquad
|a_{n}|  \geq |\eta_{n}| \inf_{K} dH(Y)>0.
$$
This means that there is a convergent subsequence (which we will denote in the same way), such that
$\lim_{k\to\infty}\eta_{k} =\eta$.
Knowing that $(v_{k},\eta_{k})$ are critical points, we get
$$
v_{k}(t) \in K \qquad \forall\ k\in \mathbb{N},\ t \in S^{1},\qquad
|\partial_{t}v_{k}(t)|  \leq |\eta_{k}|\ \sup_{K}|X_H|,
$$
hence we can use Arzel\'{a}-Ascoli theorem, to conclude that there exists a convergent subsequence satisfying
$$
\lim_{k\to\infty}(v_{k},\eta_{k})=(v,\eta).
$$
By the fact that ${\cA^{H}}$ is $C^{2},\ (v,\eta)\in {\crit}({\cA^{H}})$ and ${\cA^{H}}(v,\eta)=a$ . This proves that ${\critval}({\cA^{H}})$ is closed.

Suppose that $0$ were an accumulation point of ${\critval}({\cA^{H}})$. Then by the argument above, we would have a sequence
$$
(v_{k},\eta_{k})\in {\crit}({\cA^{H}})\cap ({\cA^{H}})^{-1}([ \inf_{k\in\mathbb{N}}a_{k},\sup_{k\in\mathbb{N}}a_{k} ]\setminus\{0\}), \qquad\textrm{such that} \qquad\lim_{k\to\infty}\eta_{k} =0.
$$
But this would contradict the results of Lemma \ref{lem:etaInf}.
Therefore $0$ is not an accumulation point.

Now we would like to prove that ${\critval}({\cA^{H}})\setminus\{0\}$ has no accumulation points. We can argue as in \cite[Thm. 23]{Fauck2016}, and apply the same argument as in the proof of Lemma \ref{lem:etaInf} to other connected components of ${\crit}({\cA^{H}})$, to show that the Morse-Bott property of ${\cA^{H}}$ guarantees that ${\crit}({\cA^{H}})$ is closed and all its connected components are iolated, even if $ {\crit}({\cA^{H}})\setminus \Sigma_0$ consists of manifolds of different dimension (as is the case in \cite{Fauck2016}).

Consequently, the corresponding critical values are also isolated.
\end{proof}

We will conclude this section with the proof of Theorem \ref{thm:MB(K)}.

\noindent\textit{Proof of Theorem \ref{thm:MB(K)}}: By uniform continuity of \ref{PO+} at $H$, there exists an open neighborhood $B(H)$ of $0$ in $C_c^\infty(M)$ and an exhaustion $\{K_{n}\}_{n\in\mathbb{N}}$ of $M$ by compact sets $K_{n}$, such that for every $n\in \mathbb{N}$ and every $h\in B(K_{n})=B(H)\cap C_c^\infty(K_{n})$ all the nondegenerate periodic orbits of the perturbed Hamiltonian $H+h$, are contained in $K_{n}$. 
By applying \cite[Thm. B.1]{CieliebakFrauenfelder2009}, we can conclude that for every $n\in \mathbb{N}$ there exists a comeager subset $A(K_n)\subseteq B(K_n)$, such that for every $h\in A(K_n)$ the following properties hold:
\begin{enumerate}
\item the Rabinowitz action functional $\cA^{H+h}$ is Morse-Bott;
\item the critical set of $\cA^{H+h}$
consists of the hypersurface $(H+h)^{-1}(0)\times\{0\}$ and a disjoint union of circles. In particular, $\eta$ is constant on every connected component of ${\crit}(\cA^{H+h})$. 
\end{enumerate}
By Lemma \ref{lem:MBtype} the two conditions above imply that for all $h\in A(K_n)$, the non-degenerate periodic orbits of the Hamiltonian flow of $X_{H+h}$ are of Morse-Bott type. For every $n$ the space $C^{\infty}_{c}(K_n)$ is a Baire space, 
hence $A(K_n)$ is dense in $B(K_n)$. This means that the union $\displaystyle{\cup_{n\in\mathbb{N}}A(K_n)}$ is dense in $B(H)$. Since $\displaystyle{\cup_{n\in\mathbb{N}}A(K_n)}$ is a subset of $A(H)$, the claim follows.

Moreover, from \cite[Thm. B.1]{CieliebakFrauenfelder2009} we can infer that for $h\in \displaystyle{\cup_{n\in\mathbb{N}}A(K_n)}$  the corresponding ${\crit}(\cA^{H+h})\setminus (H+h)^{-1}(0)\times\{0\}$ is a $1$-dimensional submanifold in $C^{\infty}(S^{1},M)$. In particular all of its connected components are isolated, which implies that $a \in {\critval}({\cA^{H+h}}),\ a\neq 0$ cannot be an accumulation point of ${\critval}({\cA^{H+h}})$.

Now suppose that $\Sigma$ is of contact type. By possibly shrinking $B(H)$ we can assume that for every $h\in B(H)$ the associated hypersurface $(H+h)^{-1}(0)$ is of contact type. Then by Lemma \ref{lem:critVal}, for every $h\in A(H)$ the corresponding set of critical values ${\critval}(\cA^{H+h})$ is closed and discrete.

\hfill $\square$

\section{Compactness}
\label{sec:compact}
In this section we will consider moduli spaces of flow lines with cascades, as defined in \cite{Frauenfelder2004}.
The setting is as follows: consider an exact, symplectic manifold $(M,\omega=d\lambda)$, a Hamiltonian function $H:M \to \mathbb{R}$, and let ${\cA^{H}}$ be the associated Rabinowitz action functional.
Fix a $2$-parameter family of $\omega$-compatible, almost complex structures $J=\{J(\cdot, \eta, t)\}_{(\eta,t)\in\R\times S^{1}}\in \mathscr{J}^\infty(M,\omega)$. Let $(f,g)$ be a Morse-Smale pair on ${\crit}({\cA^{H}})$.
For a pair of critical points $p,q \in {\crit}(f)$, we will denote by $\overline{{\mathscr{M}}}_{0}(p,q)$ the set of flow lines with zero cascades from $p$ to $q$ and by 
$\overline{{\mathscr{M}}}_{m}(p,q)$ the set of flow lines with $m$ cascades from $p$ to $q$.

The group $\mathbb{R}$ acts by timeshift on the Morse trajectories of $\overline{{\mathscr{M}}}_{0}(p,q)$ and for $m\geq 1$
the group $\mathbb{R}^{m}$ acts on $\overline{{\mathscr{M}}}_{m}(p,q)$ by time shift on each cascade: we denote the quotients by ${\mathscr{M}}_{0}(p,q)$ and ${\mathscr{M}}_{m}(p,q)$, respectively, and define the \emph{set of flow lines with cascades from} $p$ \emph{to} $q$ by
\begin{equation}
{\mathscr{M}}(p,q):= \bigcup_{m\in\mathbb{N}_{0}}
{\mathscr{M}}_{m}(p,q).
\label{M(p,q)}
\end{equation}

Since the action is increasing along the cascades, the following properties of the moduli spaces easily follow: 
\begin{itemize}
\item if ${\cA^{H}}(p)>{\cA^{H}}(q)$, then ${\mathscr{M}}(p,q)= \emptyset$,
\item if ${\cA^{H}}(p)={\cA^{H}}(q)$, then ${\mathscr{M}}(p,q)={\mathscr{M}}_{0}(p,q)$,
\item if ${\cA^{H}}(p)<{\cA^{H}}(q)$, then ${\mathscr{M}}_{0}(p,q)=\emptyset$.
\end{itemize}

In order to prove compactness of the moduli spaces of flow lines with cascades, we will need compactness both of the Morse and Floer trajectories forming the cascades. 
One result which we will repeatedly need below is Gromov compactness of the Floer trajectories, that is, convergence of the Floer trajectories in the in $C^{\infty}_{loc}$ sense.
Thanks to the bounds on the moduli spaces of Floer trajectories, the analogue of \cite[Prop. 3b]{Floer1989} holds, as one can prove by standard arguments (cf. for instance \cite[Thm. 6.5.4]{Audin2014}), after a suitable adaptation to the Rabinowitz Floer setting. The proof is even slightly simplified by the fact that our ambient symplectic manifold is exact, therefore bubbling and "cusps" can be excluded a priori.

\subsection{The "shade" of an action window}

In order to prove that the moduli spaces of cascades can be compactified, we will need the following result on moduli spaces of Floer trajectories.
The setting is as follows: we consider an exact symplectic manifold $(M,\omega)$ with a Hamiltonian $H$ and an $\omega$-compatible almost complex structure $J$. On the moduli spaces of Floer trajectories associated to $H$ and $J$ we define evaluation maps:
\begin{equation}
\ev^-(u):=\lim_{s\rightarrow -\infty} u(s,t)\quad \textrm{and}\quad \ev^+(u):=\lim_{s\rightarrow +\infty} u(s,t). \label{eqn:ev}
\end{equation}
For a fixed value $b\in \mathbb{R},\ b>0$, define
 $$
\tilde{K}(b):=\Big(\bigcup_{\substack{\Lambda \subseteq  {\crit}({\cA^{H}}),\\ {\cA^{H}}(\Lambda)\in (0,b]}}\ev^{-}(\mathscr{M}(\Sigma_0,\Lambda))\Big),
$$
where $\mathscr{M}(\Sigma_0,\Lambda)$ stands for the moduli space of the Floer trajectories between connected components of $\crit(\cA^H)$, namely $\Sigma_0= H^{-1}(0)\times\{0\}$ and $\Lambda$ (see Definition \ref{moduli}). The set $\tilde{K}(b)$ consists of limit points of trajectories between connected components of the critical set in the action window $(0,b]$ and $\Sigma_0$, and can be thought of as the \emph{shade} of moduli spaces with action within $0$ and $b$ on $\Sigma_0$.

\begin{figure}[!htb]
\begin{center}
\def\svgwidth{\columnwidth}
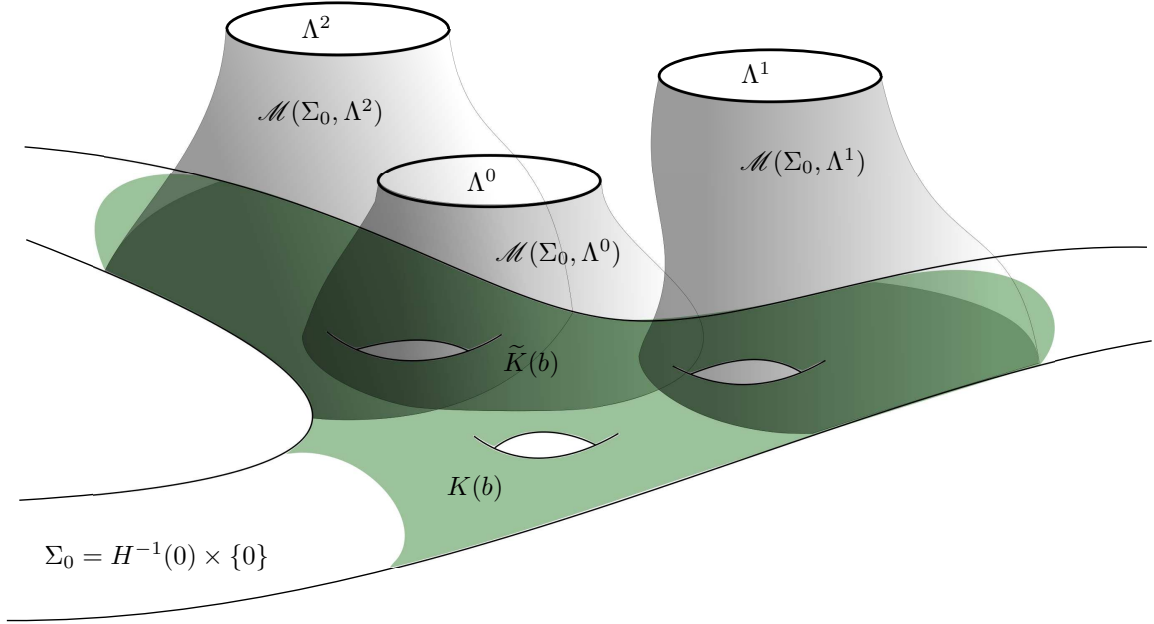
\end{center}
\caption{Construction of the sets $\tilde{K}(b)$ and $K(b)$.}
\label{fig:shade}
\end{figure}

\begin{lemma}\label{compact_shade}
Suppose the Hamiltonian $H$ satisfies the assumptions of Lemma \ref{lem:critVal} 
and admits an $\omega$-compatible almost complex structure $J\in \mathscr{J}^\infty(M,\omega)$, such that for every fixed action window, the moduli spaces of Floer trajectories between components of the critical set of $\cA^H$ within this action window are uniformly bounded in the $L^\infty$-norm.
Then for every $b>0$ the corresponding set $\tilde{K}(b)$ is a compact subset of $\Sigma_0$. In particular, it follows that for any coercive function $f$ on $\Sigma_0$ the set
\begin{equation}
K(b):= f^{-1}((-\infty,\max_{\tilde{K}(b)} f ])
\label{eqn:K(b)}
\end{equation}
is also a compact subset of $\Sigma_0$.
\end{lemma}
In other words, the set $\tilde{K}(b)$ is obtained by evaluating the endpoints of moduli spaces with action between $0$ and $b$ onto hypersurface $\Sigma_0$, whereas the set $K(b)$ is the minimal sublevel set of $f$ containing $\tilde{K}(b)$: see Figure \ref{fig:shade}.
\begin{proof}
Let $\{x_{k}\}_{k\in \mathbb{N}}$ be a sequence in $\tilde{K}(b)$. By assumption the Hamiltonian $H$ satisfies the conditions of Lemma \ref{lem:critVal}, hence the number of connected components of
$$
{\crit}({\cA^{H}})\cap({\cA^{H}})^{-1}((0,b])
$$
is finite. Therefore, without loss of generality we can assume that there exist a connected component
$\Lambda \subseteq {\crit}({\cA^{H}})\cap({\cA^{H}})^{-1}((0,b])$
and a sequence $\{u_{k}\}_{k\in \mathbb{N}}$ corresponding to $\{x_{k}\}_{k\in \mathbb{N}}$, such that
$$
u_{k}  \in \mathscr{M}(\Sigma_0,\Lambda)\qquad\textrm{and} \qquad
x_{k}  = \ev^{-}(u_{k}).
$$
By assumption the image of $\mathscr{M}(\Sigma_0,\Lambda)$ is contained in a compact subset of $M\times\mathbb{R}$, which only depends on $b$. Therefore we can apply standard compactness arguments (see for example \cite[Prop. 6.6.2]{Audin2014}) to deduce that there exists a (possibly different) connected component 
$\widetilde{\Lambda} \subseteq {\crit}({\cA^{H}})\cap({\cA^{H}})^{-1}((0,b])$ and a sequence $s_{k}\in\mathbb{R}$, which admits a subsequnce convergin in the $C^{\infty}_{loc}$ sense:
$$
u_{k}(s_{k}+\cdot) \xrightarrow[k\to \infty]{C^{\infty}_{loc}} u \in \mathscr{M}(\Sigma_0,\widetilde{\Lambda}).
$$

Moreover, by the bounds on the moduli spaces of Floer trajectories, $\ev^{-}(\mathscr{M}(\Sigma_0,\Lambda)$ is contained in a compact subset of $\Sigma$. By assumption, ${\cA^{H}}$ is Morse-Bott and the moduli spaces $\mathscr{M}(\Sigma_0,\Lambda)$ and $\mathscr{M}(\Sigma_0,\widetilde{\Lambda})$ are bounded in the $L^\infty$-norm,
hence by \cite[Thm. 25]{Fauck2016}, there exist $\alpha,\delta, \tilde{\alpha},\tilde{\delta}>0$ corresponding to $\Lambda,\widetilde{\Lambda}$ respectively, such that for every $\varepsilon>0$ we can choose $s_{0}\in\mathbb{R}$ big enough to satisfy
\begin{align*}
\dist(u_{k}(s_{k}+s_{0},t), \ev^{-}(u_{k})) & \leq \alpha e^{-\delta s_{0}}< \frac{\varepsilon}{3}\qquad\forall\ k\in \mathbb{N},\\
\dist(u(s_{0},t), \ev^{-}(u)) & \leq \tilde{\alpha} e^{-\tilde{\delta} s_{0}}< \frac{\varepsilon}{3}.
\end{align*}
Now, since $u_{k}(s_{k}+\cdot)  \xrightarrow[k\to \infty]{C^{\infty}_{loc}} u $, there exists $N\in \mathbb{N}$, such that for all $k\geq N$
$$
\dist(u_{k}(s_{k}+s_{0},t), u(s_{0},t)) < \frac{\varepsilon}{3}.
$$
Combining the above inequalities, we obtain that for every $\varepsilon>0$ there exists $N\in \mathbb{N}$, such that for all $k\geq N$
$$
\dist(\ev^{-}(u), \ev^{-}(u_{k}))<\varepsilon.
$$
This ensures that $\{x_{k}\}_{k\in \mathbb{N}}$ has a converging subsequence and the limit is in fact in $\tilde{K}(b)$, which proves the first claim. The compactness of $K(b)$ follows directly from its definition and the coerciveness of $f$.
\end{proof}

\subsection{Compactness of the moduli spaces of cascades}
We are now ready to prove that in the presence of $L^\infty$-bounds on the moduli spaces of Floer trajectories (as we obtained for tentacular Hamiltonians), the moduli spaces of cascades can be compactified by considering broken flow lines with cascades: these flow lines and the corresponding notion of Floer-Gromov convergence are defined in \cite[Def A.5, Def. A.9]{Frauenfelder2004}.

\begin{theorem}
Suppose the Hamiltonian $H$ satisfies the assumptions of Lemma \ref{compact_shade} and let $(f,g)$ be a Morse-Smale pair on ${\crit}(\cA^H)$, such that $f$ is coercive on $\Sigma$.
Fix $p,q\in {\crit}(f)$ and let $\{\mathbf{u}^{n}\}_{n\in \mathbb{N}}$ be a sequence of flow lines with cascades in the moduli space ${\mathscr{M}}(p,q)$. Then there exist a subsequence $\{\mathbf{u}^{n_{j}}\}_{j\in \mathbb{N}}$ and a broken flow line with cascades $\{\mathbf{w}_{j}\}_{j=1}^{l}$, such that $\{\mathbf{u}^{n_{j}}\}_{j\in \mathbb{N}}$ Floer-Gromov converges to $\{\mathbf{w}_{j}\}_{j=1}^{l}$.
\label{twr:compcas}
\end{theorem}

\begin{proof}
The scheme of the proof is the same as in the compact case: the compactness of the Floer trajectories follows from uniform bounds on the moduli spaces and standard compactness arguments (see \cite[Prop. 6.6.2]{Audin2014} and Section 4.2 in \cite{bourgeois2009}). Compactness of the spaces of Morse trajectories on the critical set and in particular on the non-compact hypersurface $\Sigma$ needs both the uniform bounds on the moduli spaces of Floer trajectories and the coercivity of $f$.

Let ${\cA^{H}}(p)=a$ and ${\cA^{H}}(q)=b$.
First observe that without loss of generality we can assume that all the connected components of ${\crit}({\cA^{H}})$ with the only exception of $\Sigma$ are compact, hence the restriction of the Morse flow to those components is compact up to breaking (cf. \cite[Lem. 2.38]{schwarz1993}). Therefore we only have to consider the case where $0\in [a,b]$ and show that the gradient flow of $f$ on $\Sigma$, which is part of the flowlines with cascades, is in fact restricted to compact sets.

We consider two cases:
\begin{enumerate}
\item $b=0$: then the Morse parts of the cascades in ${\mathscr{M}}(p,q)$ on $\Sigma$ are contained in $W^{s}_f(q)$. The set $W^{s}_f(q)$ is a subset of $f^{-1}((-\infty,f(q)])$, which by coercivity of $f$ is a compact subset of $\Sigma_0$. 
\item $0\in [a,b)$: then by definition of $K(b)$ in (\ref{eqn:K(b)}) the Morse components of the cascades in ${\mathscr{M}}(p,q)$ passing through $\Sigma$ are contained in $K(b)$, which by Lemma \ref{compact_shade} is compact.
\end{enumerate}

Knowing that the Morse parts of the flow lines with cascades are confined to compact subsets of $\Sigma$, we can apply the same arguments as in the compact case \cite{CieliebakFrauenfelder2009} to infer compactness of the Morse trajectories. Having established compactness both of the Morse trajectories and of the Floer trajectories, we can now directly apply the methods from \cite{Frauenfelder2004}, Appendix C.2, to show the compactness of the cascades.

\end{proof}
\section{Transversality}

The aim of this section is to prove that moduli spaces of flow lines with cascades form smooth manifolds. In order to achieve this, we first consider spaces of Floer trajectories between different components of the critical set of $\cA^H$. The following remark is in place: in the rest of this section, we will assume that the exact symplectic manifold we consider satisfies the condition $c_1(M)=0$, which guarantees that the Conley-Zehnder index is integer-valued for all the periodic orbits and therefore the Rabinowitz Floer homology is defined with integer grading. This case suffices for us, since in the end we are going to concentrate on a class of Hamiltonians defined on the standard symplectic space $(\mathbb{R}^{2n},\omega_0)$, which obviously satisfies the condition. 
If $c_1(M)|_{\pi_2(M)}\neq 0$, one can still define Rabinowitz Floer homology as a $\mathbb{Z}_2$-graded invariant. For a discussion of the grading in that case we refer the reader to \cite{hofer1995}.
 
\begin{theorem}
Consider a smooth Hamiltonian $H$ on an exact symplectic manifold $(M,\omega)$, such that $\Sigma=H^{-1}(0)$ is a regular level set and a smooth hypersurface of exact contact type. Assume $H$ satisfies \ref{PO} and \ref{MB} and there exists an open subset $\mathcal{V}$ of $M\times \mathbb{R}$, which contains the images of all the non-degenerate critical points of $\cA^H$.
Additionally, assume that for any $J\in \mathscr{J}(M,\omega,\mathcal{V})$ and any pair $\Lambda^{-}$ and $\Lambda^{+}$ of connected components of $\crit(\cA^H)$ the associated moduli space of Floer trajectories is bounded in $L^\infty$-norm.

Then for a generic choice of $J \in \mathscr{J}(M,\omega,\mathcal{V})$ and every pair of connected components without boundary $\Lambda^{-},\Lambda^{+}\subseteq {\crit}({\cA^{H}})$ the associated moduli space $\mathscr{M}_{J,H}(\Lambda^{-},\Lambda^{+})$ is a smooth manifold without boundary of dimension
$$
\dim\left(\mathscr{M}_{J,H}(\Lambda^{-},\Lambda^{+})\right)= \mu_{CZ}(\Lambda^{+})-\mu_{CZ}(\Lambda^{-})+\frac{1}{2}( \dim(\Lambda^{-})+\dim(\Lambda^{+})).
$$
\label{twr:trasvers}
\end{theorem}

\begin{proof}
The proof of this theorem is a direct adaptation of existing results to our framework. 
One first needs to show that there exists a residual subset $\mathscr{J}_{reg}(H,\mathcal{V})\subset \mathscr{J}(M,\omega,\mathcal{V})$ such that if $J\in \mathscr{J}_{reg}(H,\mathcal{V})$, then for every pair $\Lambda^+,\Lambda^-$ of components of ${\crit}({\cA^{H}})$, the spaces $\mathscr{M}_{J,H}(\Lambda^{-},\Lambda^{+})$ are all smooth manifolds.  

In order to be able to apply the implicit function theorem, we need to embed $\mathscr{M}_{J,H}(\Lambda^{-},\Lambda^{+})$ in a Banach manifold. Fix $p>2$. The following identity is proved in \cite[Thm. 25]{Fauck2016}: 
\begin{equation}
\mathscr{M}_{J,H}(\Lambda^{-},\Lambda^{+})=\left\lbrace u\in \mathscr{B}_{\delta}(\Lambda^{-},\Lambda^{+}) \ \big|\  \partial_{s} u =\nabla_{J}{\cA^{H}}(u)\right\rbrace, 
\label{eqn:MeqBnabA}
\end{equation}
where $\mathscr{B}_{\delta}(\Lambda^{-},\Lambda^{+})$ denotes the Banach manifold 
\[
\mathscr{B}_{\delta}(\Lambda^{-},\Lambda^{+}):= 
\left\{
\begin{array}{c|c}
 & \exists\ (x^{-},x^{+}) \in \Lambda^{-}\times\Lambda^{+}\\
u\in W^{1,p}_{loc}(\mathbb{R}\times S^{1}, M)\times W^{1,p}_{loc}(\mathbb{R}) & {\rm dist}(u(t,s),x^{\pm}(t)) \leq \alpha e^{\mp\delta s}  \\
& |\partial_{s}u(s,t)| \leq \alpha e^{-\delta|s|}
\end{array} \right\}.
\]
Define a Banach bundle $\mathscr{E}_{\delta}\to \mathscr{B}_{\delta}(\Lambda^{-},\Lambda^{+})$ by requiring the fiber over $u\in \mathscr{B}_{\delta}(\Lambda^{-},\Lambda^{+})$ to be
\[
\big(\mathscr{E}_{\delta}\big)_{u}:= L^{p}_{\delta}(\mathbb{R}\times S^{1}, u^{*}(TM\times \mathbb{R})).
\]
For $l\in \mathbb{N}$, let $\mathscr{J}^{l}(M,\omega,\mathcal{V})$ be defined in the same way as $\mathscr{J}(M,\omega,\mathcal{V})$, with the only difference that we require $J$ to be of class $C^{l}$ instead of class $C^{\infty}$. 
For a fixed $J\in \mathscr{J}^{l}(M,\omega,\mathcal{V})$, we consider the section $\overline{\partial}_{J} : \mathscr{B}_{\delta}(\Lambda^{-},\Lambda^{+})  \to \mathscr{E}_{\delta}$ defined
$$
\overline{\partial}_{J}(v,\eta)  := 
\left(
\begin{array}{c}
\partial_{s}v + J_{t}(v,\eta)(\partial_{t}v-\eta X_{H}(v))\\
\partial_{s}\eta +\int_{S^{1}}H(v)dt
\end{array} \right),
$$
and extend it to a section
\[
S : \mathscr{B}_{\delta}(\Lambda^{-},\Lambda^{+})  \times \mathscr{J}^{l}(M,\omega,\mathcal{V}) \to \mathscr{E}_{\delta},\qquad
S(u,J)=\overline{\partial}_{J}(u).
\]
One needs to prove that for $(u,J)$ in $S^{-1}(0)$, the vertical derivative 
\[
D_{(u,J)}S: W^{1,p}_{\delta}(\mathbb{R}\times S^{1}, u^{*}(TM \times \mathbb{R}))\times T_{J}\mathscr{J}^{l}(M,\omega,\mathcal{V}) \to \mathscr{E}_{\delta},
\]
is surjective or, equivalently, any element of the dual space $\big(\mathscr{E}_{\delta}\big)_{u}^{*}$ vanishing on the image of $D_{(u,J)}S$ is, in fact, $0$.

Following the proofs of Theorem 4.10 and 4.11 in \cite{AbbonMerry2014}, we introduce the following sets for a Floer trajectory $u\in \mathscr{M}_{J,H}(\Lambda^-,\Lambda^+),\ u=(v,\eta)$:
\[
\mathcal{R}(u) :=\left\lbrace\begin{array}{c|c}
& \partial_s v(s,t)\neq 0, \quad \partial_s\eta(s)\neq 0,\\
(s,t)\in \R\times S^1 & u(s,t)\neq \lim_{s\to\pm}u(s,t),\\
& u(s,t)\neq u(s',t)\quad \forall\ s'\in \R\setminus\{s\}.
\end{array}
 \right\rbrace,
 \]
 \[
\Omega(u) := \left\lbrace\begin{array}{c|c}
(s,t)\in \mathcal{R}(u) & v(s,t) \notin {\crit}(H),\quad v(s,t)\in \mathcal{V}
\end{array}
\right\rbrace.
\]

We would like to show that $\Omega(u)$ is open and non-empty. Without loss of generality we can assume ${\cA^{H}}(\Lambda^{+})\neq 0$. 
By assumption, $\mathcal{V}$ contains all the non-degenerate critical points of $\cA^H$, thus
\[
\lim_{s\to+\infty}u(s)\subseteq \mathcal{V}\cap (\Sigma \times \mathbb{R}).
\]
By assumption, $\mathcal{V}$ is open and $\Sigma$ is a regular level set of $H$, hence
$u^{-1}(\mathcal{V}\setminus ({\crit}(H)\times\mathbb{R}))$
is open and non-empty. On the other hand, by \cite[Thm. 4.10]{AbbonMerry2014}, the set $\mathcal{R}(u)$ is an open and dense subset of 
\[
\mathcal{C}(v):=\{(s,t)\in\mathbb{R}\times S^{1}\ |\ \partial_{s}v(s,t)\neq 0\}.
\]
Observe that $\mathcal{C}(v)$ is open, as a result of the continuity of $\partial_{s}v$. Therefore, to show that $\Omega(u)$ is open and non-empty, it suffices to show that
\begin{equation}
 u^{-1}(\mathcal{V}\setminus ({\crit}(H)\times\mathbb{R})) \cap \mathcal{C}(v)\neq \emptyset.
\label{eqn:uC(v)}
\end{equation}
Since $(v,\eta)$ is a non-constant solution of the Rabinowitz Floer equation, we can apply a generalized version of Aronszajn’s unique continuation theorem \cite[Prop. 3.3]{BourOanc2010}, and conclude that the interior of the complement of $\mathcal{C}(v)$ in $\mathbb{R}\times S^{1}$ is empty. It follows that $\mathcal{C}(v)$ is dense in $\mathbb{R}\times S^{1}$ and therefore (\ref{eqn:uC(v)}) is satisfied. Combining (\ref{eqn:uC(v)}) with the fact that $\mathcal{R}(u)$ is open and dense in $\mathcal{C}(v)$, we can finally prove that $\Omega(u)$ is open and non-empty, as claimed.

Summarizing: since every element $\rho$ of the dual space $(\mathscr{E}_{\delta})_{u}^{*}$ vanishing on the image of $D_{(u,J)}S$ is in the kernel of a first order elliptic operator and vanishes on a non-empty, open set $\Omega(u)$, we conclude (once more by Aronszajn’s unique continuation theorem) that $\rho=0$ everywhere, thus proving that $D_{(u,J)}S$ is surjective for all $(u,J)\in S^{-1}(0)$ and $S^{-1}(0)$ is a smooth manifold of class $C^{l}$. Note that all the connected components $\Lambda^{-},\Lambda^{+}\subseteq {\crit}({\cA^{H}})$  are open sets, therefore $S^{-1}(0)$ is in fact a smooth manifold without boundary.

Denote the set of regular almost complex structures by
$$
\mathscr{J}^{l}_{reg}(\Lambda^{-},\Lambda^{+}):=\Big\{ J\in \mathscr{J}^{l}(M,\omega,\mathcal{V})\ \Big|\ (\mathscr{B}_{\delta}(\Lambda^{-},\Lambda^{+}) \times\{J\} )\cap S^{-1}(0) \textrm{ is a } C^{l} \textrm{ manifold}\Big\}.
$$
Since $S^{-1}(0)$ is a manifold of class $C^{l}$, the set of regular values of the projection
$$
\pi: S^{-1}(0) \to \mathscr{J}^{l}(M,\omega,\mathcal{V}),
$$
is in fact equal to $\mathscr{J}^{l}_{reg}(\Lambda^{-},\Lambda^{+})$. By Sard's theorem (see \cite[Thm. A.5.1]{McDuff.Salamon.2012}), $\mathscr{J}^{l}_{reg}(\Lambda^{-},\Lambda^{+})$ is residual in $\mathscr{J}^{l}(M,\omega,\mathcal{V})$ with respect to the $C^{l}$-topology.

By assumption $\Sigma$ is of contact type and $H$ satisfies \ref{PO+} and \ref{MB}, hence by Lemma \ref{lem:critVal} the set of connected components of ${\crit}({\cA^{H}})$ is countable. Therefore, if we consider the intersection over the set of all connected components of ${\crit}({\cA^{H}})$, we obtain that
$$
\mathscr{J}_{reg}^{l}(H,\mathcal{V}) := \bigcap_{\Lambda^{-},\Lambda^{+} \subseteq {\crit}({\cA^{H}})}
 \mathscr{J}_{reg}^{l}(\Lambda^{-},\Lambda^{+}),
$$
is a countable intersection of residual sets, and thus a residual set in $\mathscr{J}^{l}(M,\omega,\mathcal{V})$. Finally, one can get a result in the smooth category, by using Taubes' trick to show that the intersection over all $l\in \mathbb{N}$
$$
\mathscr{J}_{reg}(H,\mathcal{V}):=\bigcap_{l\in\mathbb{N}} \mathscr{J}^{l}_{reg}(H,\mathcal{V}),
$$
is in fact residual in $\mathscr{J}(M,\omega,\mathcal{V})$. 

The Conley-Zehnder index is invariant under homotopies, hence it is constant on the connected components of ${\crit}({\cA^{H}})$. Let $J\in \mathscr{J}_{reg}(H,\mathcal{V})$ and let $\Lambda^{-},\Lambda^{+}$ be a pair of connected components of ${\crit}({\cA^{H}})$. Choose $(v^{\pm},\eta^{\pm})\in\Lambda^{\pm}$ with cappings $\bar{v}^{\pm}\in C^{\infty}(D,M)$, such that $\bar{v}^{\pm}(e^{i\pi t})=v^{\pm}(t)$ and let $u=(v,\eta)$ be a Floer trajectory such that $\lim_{s\to\pm\infty}u(s)=(v^{\pm},\eta^{\pm})$. Then by \cite[Prop. 4.1]{CieliebakFrauenfelder2009} the local virtual dimension of the moduli space at $u$ can be expressed in terms of the Conley-Zehnder indices of the endpoints and the first Chern class of the sphere obtained by capping the cylinder $v$ with $\bar{v}^{-}$ and $\bar{v}^{+}$. More precisely the local dimension of the moduli space is equal to
$$
\mu_{CZ}(v^{+},\eta^{+})-\mu_{CZ}(v^{-},\eta^{-})+\tfrac{1}{2}( \dim(\Lambda^{-})+\dim(\Lambda^{+}))+c_{1}\left(\bar{v}^{-}\# v\# \bar{v}^{+}\right).
$$
If $c_{1}(M)|_{\pi_{2}(M)}=0$, the last term disappears, the Conley-Zehnder index does not depend on the capping or the choice of $(v^{\pm},\eta^{\pm})\in\Lambda^{\pm}$, hence we obtain that the dimension of the moduli space of Floer trajectories between $\Lambda^{-}$ and $\Lambda^{+}$ is
\begin{equation}
\dim(\mathscr{M}_{J,H}(\Lambda^{-},\Lambda^{+})) = \mu_{CZ}(\Lambda^{+})-\mu_{CZ}(\Lambda^{-})+\tfrac{1}{2}( \dim(\Lambda^{-})+\dim(\Lambda^{+})).
\label{eqn:dimM}
\end{equation}
\end{proof}

We now concentrate on the smooth structure of moduli spaces of flow lines with cascades.
We assume that $f$ is a Morse function on the critical manifold of $\cA^H$, and we introduce a grading on ${\crit}(f)$ in the following way.
Given $x\in {\crit}(f)$, let $\mu_{\sigma}(x)$ be the signature index with respect to $f$, namely
\begin{equation}
\mu_{\sigma}(x):=\frac{1}{2}(\textrm{dim}(W^s_f(x))-\textrm{dim}(W^u_f(x))),
\label{eqn:muSig}
\end{equation}
where $W^s_f$ and $W^u_f$ are the stable and unstable manifolds at $x$ with respect to the flow of $\nabla f$.
Let $\mu_{CZ}(x)$ be the Conley-Zehnder index with respect to ${\cA^{H}}$. We define the grading on the Rabinowitz Floer complex by 
\begin{equation}
\mu(x):=\mu_{\sigma}(x)+\mu_{CZ}(x) +\tfrac{1}{2}.
\label{muRF}
\end{equation}

\begin{proposition}
Consider a smooth Hamiltonian $H$ on an exact symplectic manifold $(M,\omega)$, such that the level set $\Sigma$ is a smooth hypersurface of exact contact type. Additionally, assume that $H$ satisfies \ref{PO} and \ref{MB} and there exists $J\in\mathscr{J}(M,\omega)$, such that for every pair of connected submanifolds without boundary $\Lambda^{-},\Lambda^{+}\subseteq {\crit}({\cA^{H}})$ the associated moduli space $\mathscr{M}_{J,H}(\Lambda^{-},\Lambda^{+})$ is a smooth manifold without boundary, whose image in $M\times \mathbb{R}$ is bounded.

Then for a Morse-Smale pair $(f,g)$ on ${\crit}({\cA^{H}})$ the associated moduli space ${\mathscr{M}}(x^{-},x^{+})$ of flowlines with cascades is a smooth, finite dimensional manifold without boundary. Its dimension is given by 
\[
\dim({\mathscr{M}}(x^{-},x^{+}))= \mu (x^{+})-\mu (x^{-})-1.
\]
\label{prop:casmfld}
\end{proposition}

\begin{proof}
The first step is to show that for every pair $x^{-},x^{+}\in {\crit}(f)$ and every $m\in\mathbb{N}$ the corresponding moduli space ${\mathscr{M}}_{m}(x^{-},x^{+})$ is a smooth manifold. 

For $m=0$ the points $x^{-}$ and $x^{+}$ belong to the same connected component of ${\crit}({\cA^{H}})$ and the moduli space ${\mathscr{M}}_{0}(x^{-},x^{+})$ consists of Morse flowlines. By assumption the pair $(f,g)$ was chosen to be Morse-Smale, hence by standard Morse-theoretical arguments \cite{schwarz1993} we conclude that ${\mathscr{M}}_{0}(x^{-},x^{+})$ is a smooth manifold of dimension $\mu_{\sigma}(x^{+})-\mu_{\sigma}(x^{-})-1$. 
In this case $x^{-}$ and $x^+$ belong to the same connected component of ${\crit}(\cA^{H})$, so their Conley-Zehnder indeces are the same, which implies that the dimension of ${\mathscr{M}}_{0}(x^{-},x^{+})$ is in fact equal to $\mu(x^{+})-\mu(x^{-})-1$.

For $m\geq 1$ whenever ${\mathscr{M}}_{m}(x^{-},x^{+})\neq \emptyset$ then the points $x^{-}$ and $x^{+}$ belong to different critical components of ${\crit}(\cA^{H})$. Following the arguments from the Appendix of \cite{Frauenfelder2004} and \cite{Fauck2016}, section 2.4, we conclude that ${\mathscr{M}}_{m}(x^{-},x^{+})$ is a smooth manifold of dimension $\mu(x^{+})-\mu(x^{-})-1$ with boundary. By Theorem \ref{twr:compcas} if we take the union of the spaces ${\mathscr{M}}_{m}(x^{-},x^{+})$ over all $m \geq 1$ (cf. \eqref{M(p,q)}), we obtain a space ${\mathscr{M}}(x^{-},x^{+})$ which is precompact. In fact we can glue the manifolds with different number of cascades along the boundary (cf. \cite[Prop. 1b]{Floer1989}, \cite[Cor. A.15]{Frauenfelder2004}, \cite[Thm. 9.2.3]{Audin2014}) to equip the set ${\mathscr{M}}(x^{-},x^{+})$  with a smooth structure. As a result ${\mathscr{M}}(x^{-},x^{+})$ becomes a smooth, precompact manifold without boundary, of dimension $\mu(x^{+})-\mu(x^{-})-1$.
\end{proof} 
\section{Definition of RFH}
\label{sec:defRFH}

In this section we will extend the definition given in \cite{CieliebakFrauenfelder2009} so that it also applies to a class of non-compact hypersurfaces. 

\begin{theorem}
Consider a smooth Hamiltonian $H$ on an exact symplectic manifold $(M,\omega)$, such that $\Sigma=H^{-1}(0)$ is a regular level set and a smooth hypersurface of exact contact type. Additionally, assume that $H$ satisfies \ref{PO} and \ref{MB} and there exists $J\in\mathscr{J}(M,\omega)$, such that for every pair of connected submanifolds without boundary $\Lambda^{-},\Lambda^{+}\subseteq {\crit}({\cA^{H}})$ the associated moduli space $\mathscr{M}_{J,H}(\Lambda^{-},\Lambda^{+})$ is a smooth manifold without boundary, whose image in $M\times \mathbb{R}$ is bounded. Then the associated Rabinowitz Floer homology $RFH(H,J),$ is well defined.
\label{thm:defRFH}
\end{theorem}
\begin{proof}
The proof follows the construction of Rabinowitz Floer homology in \cite{CieliebakFrauenfelder2009} and combines the results of the previous sections into a generalization for the non-compact setting.

Fix a Morse-Smale pair $(f,g)$ on ${\crit}({\cA^{H}})$, such that $f$ is coercive. Consider the chain complex $CF_*({\cA^{H}},f)$ generated by the critical points of $f$, with coefficients in $\mathbb{Z}_2$ and grading given by the index $\mu$. More precisely, $CF_{k}({\cA^{H}},f)$ consists of formal sums 
$$
p= \sum_{\substack{x\in {\crit}(f),\\ \mu(x)=k}}p_{x}x,
$$
where the coefficients $p_{x}\in \mathbb{Z}_{2}$ satisfy the following Novikov finiteness condition: for all $a\in \mathbb{R}$,
\begin{equation}
\#\{x\in {\crit}(f)\ |\ \mu(x)=k,\ p_{x}\neq 0,\ \textrm{and}\ {\cA^{H}}(x)\geq a \} < \infty.
\label{eqn:fin}
\end{equation}

By Proposition \ref{prop:casmfld}, for a given a pair of points $x,y\in {\crit}(f)$ with $\mu (x)-\mu (y)=1$, the moduli space ${\mathscr{M}}(y,x)$ is a discrete set of points. On the other hand, by Theorem \ref{twr:compcas}, ${\mathscr{M}}(y,x)$ is also compact, hence it is a finite set. Denote its modulo $2$ cardinality by 
\begin{equation}
n(y,x):=\# {\mathscr{M}}(y,x)\, \textrm{mod}\ 2\ \in \mathbb{Z}_{2},
\label{eqn:n}
\end{equation}
and define the Floer boundary operator
$$
\partial: CF_{*+1}({\cA^{H}},f) \to CF_{*}({\cA^{H}},f),
$$
as the linear extension of
$$
\partial x := \sum_{\substack{y\in {\crit}(f),\\ \mu(x)-\mu(y)=1}}n(y,x)y.
$$
In order for the boundary operator to be well defined, one has to check that the finiteness condition (\ref{eqn:fin}) holds  for $\partial x$. In other words, we must show that for all $x\in {\crit}(f)$ and $a \in \mathbb{R}$
\begin{equation}
\# \{ y \in {\crit}(f)\ |\ \mu(x)-\mu(y)=1,\ n(y,x)\neq 0\ \textrm{and}\ {\cA^{H}}(y) \geq a\} < + \infty,
\label{eqn:finit}
\end{equation}
Set $b:={\cA^{H}}(x)$. Note that $n(y,x)\neq 0$ implies $ b\geq {\cA^{H}}(y)\geq a$. 
We consider the following two cases:
\begin{enumerate}
\item If $\cA^{H}(x)=\cA^{H}(y)=0$, then $n(y,x)\neq 0$ implies $y \in  {\crit}(f)\cap \operatorname{cl}( W^{s}_{f}(x))$, where $\operatorname{cl}$ stands for closure. Moreover, $\operatorname{cl}( W^{s}_{f}(x))\subseteq f^{-1}((-\infty, f(x)])$, which is compact by coercivity of $f$ on $\Sigma$, and thus $\operatorname{cl}( W^{s}_{f}(x))\cap {\crit}(f)$ is finite.
\item In the other case, i.e. if ${\cA^{H}}(y)<b$ or $b\neq 0$, then the condition $n(y,x)\neq 0$ implies that 
$$
y \in  {\crit}(f)\cap \left(\left({\cA^{H}}\right)^{-1}\left([a, b]\setminus \{0\}\right)\cup K(b)\right).
$$
Observe that the set $\left({\cA^{H}}\right)^{-1}\left([a, b]\setminus \{0\}\right)$ is a finite union of nonzero connected components of ${\crit}({\cA^{H}})$, which are all compact. Moreover, the set $K(b)$ is compact by Lemma \ref{compact_shade}. Thus their sum contains finitely many critical points of $f$ and the condition (\ref{eqn:finit}) follows.
\end{enumerate}
Now if we combine the results from Theorem \ref{twr:compcas} and Proposition \ref{prop:casmfld} with the standard gluing argument in Floer theory (see for example Proposition 2d.1 in \cite{Floer1989} or Theorem 9.2.2 in \cite{Audin2014}), then we obtain that the operator $\partial$ defined above is indeed a boundary operator, i.e. $\partial^{2}=0$. As a result, we can define Rabinowitz Floer homology of the quadruple $(H,J,f,g)$ by setting:
$$
RFH_{k}(H,J,f,g):= \frac{\operatorname{Im}(\partial_{k+1})}{\Ker (\partial_{k})}.
$$
In fact, by Theorem 8 in \cite{schwarz1993}, the homology defined in this way is independent of the choice of Morse-Smale pair $(f,g)$ as long as $f$ is coercive on ${\crit}({\cA^{H}})$. Therefore, we can denote it by $RFH(H,J)$.
\end{proof}

Just as in the compact case, Rabinowitz Floer homology has the following property:
\begin{corollary}
Whenever $\Sigma$ carries no closed characteristics, then $RFH(H)$ is well defined and isomorphic to the singular homology of $\Sigma$ shifted by half of the dimension of $M$.
\end{corollary}
In this case, the assumptions of Theorem \ref{thm:defRFH} are trivially satisfied and the Rabinowitz Floer homology is generated by the positive gradient flow of the coercive function $f$ on $\Sigma$. If we change the sign of the gradient flow and change the order in the definition of the differential, we retrieve the Morse homology of $\Sigma$ with respect to the signature grading. Indeed, the number of positive gradient flowlines flowing out of a critical point is equal to the number of negative gradient flowlines flowing into the same point. On the other hand, the Morse homology of $\Sigma$ with signature grading is isomorphic to its Morse homology shifted by half of the dimension of $M$. Finally, the Morse homology of $\Sigma$ with Morse grading is isomorphic to its singular homology (see \cite{schwarz1993} for a proof of the last isomorphism in the non-compact case):
$$
RFH_*(H,J) = RFH_*(H,J,f,g) = MH_{*+n-1} (f,g, \Sigma) = H_{*+n-1} (\Sigma).
$$
This observation allows us to use the Rabinowitz Floer homology to detect periodic orbits. Indeed, if the Rabinowitz Floer homology of $H$ is well defined and differs from the singular homology of the hypersurface $\Sigma$, then $\Sigma$ carries a closed characteristic.

The following corollary shows that Rabinowitz Floer homology is invariant under symplectomorphisms:

\begin{proposition}
Let $\varphi:(M_1,\omega_1)\to (M_2,\omega_2)$ be a symplectomorphism of exact symplectic manifolds and suppose that the pair $(H,J)\in C^{\infty}(M_2)\times\mathscr{J}(M_2,\omega_2)$ satisfies the assumptions of Theorem \ref{thm:defRFH}, so that in particular the Rabinowitz Floer homology of the pair $(H,J)$ is well defined. Then the Rabinowitz Floer homology is also well defined for the pair $\varphi^{*}(H,J)=(H\circ \varphi,D\varphi^{-1}\circ J \circ D\varphi)\in  C^{\infty}(M_1)\times\mathscr{J}(M_1,\omega_1)$ and
$$
RFH(\varphi^{*}(H,J)) \cong RFH(H,J).
$$
\label{prop:sympRFHinv}
\end{proposition}

\begin{proof}

We will show that the pair $\varphi^{*}(H,J)$ also satisfies the conditions of Theorem \ref{thm:defRFH}.

Since $\varphi$ is a symplectomorphism, if $\lambda_2$ is a primitive of $\omega_2$, then $\varphi^{*}\lambda_2$ is a primitive of $\omega_1$, hence
$\varphi$ preserves the contact type property of the hypersurface. Furthermore, $\varphi$ induces a diffeomorphism of loop spaces
$$
\varphi_{\#}: C^{\infty}(\R/\Z;M_2)\times \R  \to C^{\infty}(\R/\Z;M_1)\times \R,\quad (v,\eta) \mapsto (\varphi^{-1}\circ v,\eta).
$$
The action functionals $\mathcal{A}^{H}$ to $\mathcal{A}^{H \circ \varphi}$ differ by composition with $\varphi_{\#}$, so in particular
$d \cA^{H\circ\varphi}= d\cA^{H}\circ D\left(\varphi_{\#}\right)^{-1}$, and ultimately $(v,\eta) \in \crit\left(\cA^{H}\right)$ if and only if $\varphi_{\#}(v,\eta)\in \crit\left(\cA^{H\circ\varphi}\right)$.
In particular, for every $a,b\in\R$ there is a one-to-one correspondence between 
$$
\crit\left(\cA^{H}\right)\cap \left(\cA^{H}\right)^{-1}\left([a,b]\right)
\quad \textrm{and} \quad
\crit\left(\cA^{H \circ \varphi}\right)\cap \left(\cA^{H \circ \varphi}\right)^{-1}\left([a,b]\right).
$$
Since $\varphi$ is a diffeomorphism, pre-images of compact sets are compact, and so we can conclude that the Hamiltonian $H\circ \varphi$ satisfies \ref{PO} if and only if $H$ satisfies \ref{PO}.

Now we will show that $\varphi_\#$ preserves the Conley-Zehnder indices. Denote by $\phi^{t}$ the Hamiltonian flow of $X_H$ on $M_2$ and let $\gamma$ be a contractible, closed characteristic of $\phi^{t}$ on $\Sigma$ with period $\eta$. Let $\lambda_2$ be the primitive of $\omega_2$, such that $\xi=\Ker\left(\lambda_2\Big|_{T\Sigma}\right)$ is the contact structure on $\Sigma$ and let $\Phi:[0,\eta]\times \R^{2n-2} \to \gamma^{*}(TM_2)$ be a trivialization of $\xi$ along the closed characteristic $\gamma$. There is a splitting $T\Sigma=X_H\oplus \xi$, which is preserved by the flow $\phi_t$. Then the Conley-Zehnder index of $(\gamma(\frac{t}{\eta}),\eta) \in \crit(\mathcal{A}^H)$ is equal to the Maslov index of the path of symplectic matrices $\Gamma_1:[0,\eta]\to Sp(n-1)$ defined by
$$
\Gamma_1(t):= \Phi^{-1}(t) \circ D\phi^t \circ \Phi (0).
$$

Let us now construct the path of symplectic matrices corresponding to the closed characteristic $\varphi^{-1} \circ \gamma$ on $(\varphi )^{-1}(\Sigma)$. The Hamiltonian flow of $X_{H\circ \varphi}$ on $M_1$ is given by $\varphi^{-1} \circ \phi^{t} \circ \varphi$ and the contact structure on $\varphi^{-1}(\Sigma)$ is given by $D\phi^{-1}(\xi)$. Then $D\varphi^{-1} \circ \Phi$ is a trivialization of $D\phi^{-1}(\xi)$ along $\varphi^{-1} \circ \gamma$. We can construct the corresponding path of symplectic matrices $\Gamma_2:[0,\eta]\to Sp(n-1)$ as before
$$
\Gamma_2(t):=(D\varphi^{-1} \circ \Phi (t))^{-1}\circ  D\left( \varphi^{-1} \circ \phi^{t} \circ \varphi\right)\circ D\varphi^{-1} \circ \Phi (0)=\Gamma_1(t).
$$
It turns out that the paths of symplectic matrices corresponding to $\gamma$ and to $\varphi^{-1}\circ \gamma$ are the same, hence the Conley-Zehnder indices are the same.

We will prove next that the map $\varphi_\#$ preserves the Morse-Bott property of the Rabinowitz action functional. Using \eqref{eqn:gJ} we can define a metric $g_J$ on $C^{\infty}(\R/\Z;M_2)\times \R$ associated to the almost complex structure $J$ and a metric $g_{\varphi^* J}$ on $C^{\infty}(\R/\Z;M_1)\times \R$ associated to the almost complex structure $\varphi^* J= D\varphi^{-1} \circ J \circ D\varphi$. Observe that the push-forward of $g_J$ to $C^{\infty}(\R/\Z;M_1)\times \R$ via $\varphi_{\#}$ is equal to $g_{\varphi^* J}$
$$
\left(\varphi_{\#}\right)^* g_J =  g_{\varphi^* J}.
$$
As a result
\begin{align}
\nabla_{g_{\varphi^* J}}\cA^{H \circ \varphi}\left( \varphi_{\#}(v,\eta)\right) & =D \varphi_{\#} \left( \nabla_{g_{J}}\cA^{H}\left( v,\eta\right)\right),\label{NabGVarJ} \\ 
\nabla_{g_{\varphi^* J}}^2 \cA^{H \circ \varphi}\left( \varphi_{\#}(v,\eta)\right) & = D \varphi_{\#} \circ  \nabla_{g_{J}}^2\cA^{H}\circ  D \left(\varphi_{\#}\right)^{-1} \nonumber
\end{align}
Consequently, if $\cA^H$ is Morse-Bott, then $\cA^{H\circ \varphi}$ is Morse-Bott, since for $(v,\eta) \in \crit \left(\cA^{H\circ \varphi}\right)$ we have
\begin{align*}
\Ker\left(\nabla_{g_{\varphi^* J}}^2 \cA^{H \circ \varphi}(v,\eta) \right)  & = \Ker\left(D \varphi_{\#} \circ  \nabla_{g_{J}}^2\cA^{H}\circ  D \left(\varphi_{\#}\right)^{-1}  \right)
 = \Ker\left(\nabla_{g_{J}}^2\cA^{H} \circ D \left(\varphi_{\#}\right)^{-1} \right) \\
& = \Ker\left( d \cA^{H} \circ  D \left(\varphi_{\#}\right)^{-1}  \right)
=\Ker\left( d_{(v,\eta)} \cA^{H\circ \varphi} \right)
= T_{(v,\eta)}\cA^{H\circ \varphi}.
\end{align*}

Finally, we show that there is a one-to-one correspondence between the Floer trajectories of the pairs $(H,J)$ and $\varphi^{*}(H,J)$. In other words, for every two connected components $\Lambda^{-},\Lambda^{+}\subseteq \crit\left( \cA^{H}\right)$ we have
\begin{equation}
u \in \mathscr{M}_{J,H}(\Lambda^{-},\Lambda^{+})
\quad \iff \quad 
\varphi_{\#} \circ u \in \mathscr{M}_{\varphi^{*}(J,H)}(\varphi_{\#}(\Lambda^{-}),\varphi_{\#}(\Lambda^{+})). \label{MvarJH}
\end{equation}
Indeed, if $u \in \mathscr{M}_{J,H}(\Lambda^{-},\Lambda^{+})$ then by \eqref{NabGVarJ} we have
$$
0 = D\varphi_{\#}\left(\partial_s u + \nabla_{g_{J}}\cA^{H}(u)\right)=\partial_s \left(\varphi_{\#} \circ u\right)+ \nabla_{g_{\varphi^* J}}\cA^{H \circ \varphi}\left( \varphi_{\#}\circ u\right).
$$
Since $\varphi$ is a diffeomorphism, the pre-images of compact sets are compact, and we can conclude that the moduli spaces $\mathscr{M}_{\varphi^{*}(J,H)}(\varphi_{\#}(\Lambda^{-}),\varphi_{\#}(\Lambda^{+}))$ are bounded smooth manifolds. It follows that the pair $\varphi^{*}(H,J)$ satisfies the conditions of Theorem \ref{thm:defRFH} and thus $RFH_*(\varphi^*(H,J))$ is well defined. Moreover, there is a one-to-one correspondence between $\crit(\mathcal{A}^H)$ and $\crit(\mathcal{A}^{H\circ\varphi})$, which preserves the Conley-Zehnder indices and by \eqref{MvarJH} there is a one-to-one correspondence between the Floer trajectories of the pair $(H,J)$ and the pair $\varphi^{*}(H,J)$. As a result, the corresponding Rabinowitz Floer homologies are isomorphic.
\end{proof}


\section{Invariance of RFH}
\label{sec:InvRFH}

The aim of this section is to show that Rabinowitz Floer homology is independent of the choice of almost complex structure and invariant under small compactly supported homotopies of the Hamiltonian.
In order to prove this, we first need to introduce the notion of perturbed flow lines with cascades. Let  $\Gamma=\{(H_{s},J_{s})\}_{s\in \mathbb{R}}$ be a homotopy of Hamiltonians and almost-complex structures, constant outside of $[0,1]$ and let $f_0$ and $f_1$ be coercive Morse functions on $\crit(\cA^{H_0})$ and $\crit(\cA^{H_1})$, respectively:
for any pair $(p,q)\in {\crit}(f_{0})\times {\crit}(f_{1})$ we consider the space of perturbed flow lines with cascades ${\mathscr{M}}^{\Gamma}(p,q)$. This is obtained as follows: we consider flow lines which consist of exactly one perturbed Floer trajectory, i.e. a solution of $\partial_{s}u=\nabla_{J_{s}}\cA^{H_{s}}$ with bounded energy, while all the other trajectories are Floer trajectories of $\cA^{H_{0}}$ or $\cA^{H_{1}}$. Then we quotient the space of such lines by the natural $\mathbb{R}$-action, as we did for the unperturbed trajectories in Section \ref{sec:compact}.

Proving independence of Rabinowitz Floer homology from the choice of almost complex structure and invariance under compactly supported perturbations requires uniform $L^{\infty}$ bounds along perturbed Floer trajectories associated to different families of homotopies. In the compact case those uniform bounds are obtained by a standard isolating neighborhood argument, similar to Theorem 3 from \cite{Floer1989}. Unfortunately, in the case of non-compact hypersurfaces the standard techniques are not directly applicable. In this section we will assume that all the necessary bounds hold and show how to prove invariance of Rabinowitz Floer homology consequently. Later we will apply this result to the case of tentacular Hamiltonians, a class for which we have established uniform bounds directly in \cite{pasquotto2017}.

Note that contrary to the unperturbed case, the action may not be monotonically increasing along the perturbed Floer trajectories. To deal with this phenomenon, we introduce the Novikov finiteness condition, which ensures that the action cannot decrease indefinitely along a perturbed Floer trajectory.
\begin{definition}
Let $\Gamma=\{(H_{s},J_{s})\}_{s\in \mathbb{R}}$ be a smooth homotopy of Hamiltonians and almost complex structures constant outside $[0,1]$. 
For a pair $a,b\in\mathbb{R}$ define
\begin{equation}
A(\Gamma,a,b) :=\inf \left\lbrace\begin{array}{r | l}
&\exists\  \Lambda_{i}\subseteq {\crit}(\mathcal{A}^{H_{i}})\quad i=0,1,\\
A \in (-\infty,b] & \cA^{H_{0}}(\Lambda_{0}) \geq a,\ \cA^{H_{1}}(\Lambda_{1})=A,\\
& \mathscr{M}_\Gamma(\Lambda_{0},\Lambda_{1})\neq \emptyset
\end{array}\right\rbrace,\label{eqn:Aab}
\end{equation}
\begin{equation}
B(\Gamma,a,b) :=\sup \left\lbrace\begin{array}{r | l}
& \exists\ \Lambda_{i}\subseteq {\crit}(\mathcal{A}^{H_{i}})\quad i=0,1,\\
B \in [a,+\infty) & \cA^{H_{0}}(\Lambda_{0}) = B,\ \cA^{H_{1}}(\Lambda_{1})\leq b,\\
& \mathscr{M}_\Gamma (\Lambda_{0},\Lambda_{1})\neq \emptyset
\end{array}\right\rbrace. \label{eqn:Bab}
\end{equation}
We say that the homotopy $\Gamma$ satisfies the \emph{Novikov finiteness condition} if for each pair $(a,b)\in {\critval}(\cA^{H_{0}})\times {\critval}(\cA^{H_{1}})$ the corresponding $A(\Gamma,a,b)$ and $B(\Gamma,a,b)$ are finite. 
\label{def:Nov}
\end{definition}

The first step towards constructing an isomorphism between $RFH(H_{0},J_{0})$ and $RFH(H_{1},J_{1})$ is to show that for any pair $(p,q)\in {\crit}(f_{0})\times {\crit}(f_{1})$ the corresponding moduli space $\mathscr{M}_\Gamma(p,q)$ is a smooth manifold without boundary, which can be compactified by moduli spaces of lower index. That will allow us to define a homomorphism between $RFH(H_{1},J_{1})$ and $RFH(H_{0},J_{0})$. To equip $\mathscr{M}_\Gamma(p,q)$ with a manifold structure we follow the same argument as in Theorem \ref{twr:trasvers}, 
but to achieve transversality, rather than perturbing the almost complex structures, we perturb the homotopy in its Hamiltonian component. Therefore we need uniform bounds on the moduli spaces not only for a fixed homotopy, but also for families of homotopies. Since the method for obtaining those bounds is substantially different from the analogous result in the compact case, we will show how to do it in Lemma \ref{lem:homo}.

We introduce below some notation which is necessary to formulate the lemma. For a compact subset $N\subseteq \Sigma$ we define 
\begin{equation}
\mathscr{C}({\cA^{H}},N):= \left\lbrace x\in {\crit} \left(\mathcal{A}^H\right)\ \Big|\ 0<|\mathcal{A}^H(x)| \quad \textrm{or}\quad x\in N\times\{0\}\right\rbrace.
\label{eqn:CAHY}
\end{equation}
Let $(H_{0},J_{0})$ and $(H_{1},J_{1})$ be pairs of Hamiltonians and compatible almost complex structures satisfying the assumptions of Theorem \ref{thm:defRFH}. Assume that there exists a compact subset $K\subseteq M$, such that the support of $H_{1}-H_{0}$ is contained in $K$. Let $\Gamma=\{(H_{s},J_{s})\}_{s\in \mathbb{R}}$ be a smooth homotopy satisfying (\ref{homotopy1}). For a perturbation $h\in C^{\infty}_{c}([0,1]\times K)$, we define the corresponding homotopy $\Gamma(h)$ by:
\begin{equation}
\Gamma(h):=\{(H_{s}+h_{s},J_{s})\}_{s\in \mathbb{R}}.
\label{eqn:Gh}
\end{equation}

\begin{lemma}\label{lem:homo}
Suppose there exists an open neighborhood $\mathscr{O}(\Gamma)$ of $0$ in $C_c^{\infty}([0,1]\times K)$, such that for every $h\in \mathscr{O}(\Gamma)$ and every pair of critical values $a\in {\critval}(\cA^{H_{0}})$ and $b\in  {\critval}(\cA^{H_{1}})$ the following holds:
\begin{enumerate}[label=(\roman*)]
\item $$
\sup_{h\in \mathscr{O}(\Gamma)}B(\Gamma(h),a,b)<+\infty \quad\textrm{and}\quad \inf_{h\in \mathscr{O}(\Gamma)}A(\Gamma(h),a,b)>-\infty;
$$
\item given a compact subset $N\subseteq H_{1}^{-1}(0)$ and a pair of connected components
\[
 (\Lambda_{0},\Lambda_{1}) \subseteq {\crit}(\cA^{H_{0}}) \times\mathscr{C}(\cA^{H_{1}},N), \nonumber
 \]
such that $a \leq \cA^{H_{0}}(\Lambda_{0})$ and $\cA^{H_{1}}(\Lambda_{1})\leq b$, 
the moduli space $\mathscr{M}_{\Gamma(h)}(\Lambda_{0},\Lambda_{1})$ admits a uniform bound on the energy, that is
$$
\sup_{u \in \mathscr{M}_{\Gamma(h)}(\Lambda_{0},\Lambda_{1})}\int_{\mathbb{R}}\|\partial_{s}u(s)\|^{2}_{L^2(S^1)\times \R}ds <+\infty,
$$
and its image is contained in a compact subset of $M\times \mathbb{R}$,
which only depends on $a,b$ and $N$.
\end{enumerate}
Then for a generic choice of $h\in \mathscr{O}(\Gamma)$, there exists a homomorphism 
$$
\Psi^{\Gamma(h)}: RFH(H_{1},J_{1}) \to RFH(H_{0},J_{0}).
$$
\end{lemma}
\begin{proof}
In this proof we will show how to obtain uniform bounds on the moduli spaces corresponding to the families of homotopies and later just refer to the standard Floer theoretical arguments for the actual definition of the homomorphisms.

Let us fix $h \in \mathscr{O}(\overline{\Gamma})$ and let $\Gamma(h)$ be the corresponding homotopy, as defined in (\ref{eqn:Gh}). We will start by proving compactness. Fix two Morse-Smale pairs $(f_{0},g_{0})$ and $(f_{1},g_{1})$ on ${\crit}(\cA^{H_{0}})$ and ${\crit}(\cA^{H_{1}})$, respectively, such that $f_{0}$ and $f_{1}$ are coercive. For a fixed pair $(p,q)\in {\crit}(f_{0})\times {\crit}(f_{1})$ we will show that there exists a compact subset of $M\times \mathbb{R}$ that contains all the broken trajectories with cascades between $p$ and $q$. 

Since $H_{0}$ and $H_{1}$ satisfy the assumptions of Lemma \ref{lem:critVal}, ${\crit}(H_{0})$ and ${\crit}(H_{1})$ are closed and discrete. Let $a:=\cA^{H_{0}}(p),\ b:=\cA^{H_{1}}(q)$. By assumption, for every fixed $h\in\mathscr{O}(\overline{\Gamma})$ the corresponding values $\alpha:=A(\Gamma(h),a,b)$ and $\beta:=B(\Gamma(h),a,b)$ are finite. As a result, the sets
$$
[a,\beta]\cap {\crit}\left(\cA^{H_{0}}\right),\qquad [\alpha,b]\cap {\crit}\left(\cA^{H_{1}}\right),
$$
are finite. 

We are going to consider different spaces of trajectories and argue that they all have bounded image in $M\times \R$. The first two spaces consist of unperturbed Floer trajectories between connected components of the critical set of a fixed Hamiltonian:
$$
\mathfrak{M}_{0}(a,b)  :=\bigcup_{\substack{\Lambda^{\pm}\subseteq {\crit}\left(\cA^{H_{0}}\right)\\  \cA^{H_{0}}\left(\Lambda^{\pm}\right)\in [a,\beta]}}\mathscr{M}(\Lambda^{-},\Lambda^{+}),\qquad \qquad
 \mathfrak{M}_{1}(a,b)  :=\bigcup_{\substack{\Lambda^{\pm}\subseteq {\crit}\left(\cA^{H_{1}}\right)\\ \cA^{H_{1}}\left(\Lambda^{\pm}\right)\in [\alpha,b]}}\mathscr{M}(\Lambda^{-},\Lambda^{+}),
$$
Since the pairs $(H_{0},J_{0})$ and $(H_{1},J_{1})$ satisfy the assumptions of Theorem \ref{thm:defRFH}, the images of $\mathfrak{M}_{0}(a,b)$ and $\mathfrak{M}_{1}(a,b)$ are bounded in $M\times\mathbb{R}$ (finite unions of bounded sets).

The third space consists of perturbed trajectories connecting components of the critical set of $\cA^{H_{0}}$ with components of the critical set of $\cA^{H_{1}}$:
$$
\mathfrak{M}_{2}(a,b):= \bigcup_{\substack{\Lambda_{0}\subseteq {\crit}\left(\cA^{H_{0}}\right)\\ \cA^{H_{0}}\left(\Lambda_{0}\right)\in [a,\beta]}}\ \bigcup_{\substack{\Lambda_{1}\subseteq \mathscr{C}\left(\cA^{H_{1}},N\right)\\ \cA^{H_{1}}\left(\Lambda_{1}\right)\in [\alpha,b]}}\mathscr{M}_{\Gamma(h)}(\Lambda_{0},\Lambda_{1}).
$$
Note that for a compact $N\subseteq (H_{1})^{-1}(0)$, the set $\mathfrak{M}_{2}(a,b)$ is a finite union of sets whose image in $M\times\mathbb{R}$ is by assumption (ii) bounded, so the image of $\mathfrak{M}_{2}(a,b)$ itself is also bounded.

Now we need to distinguish three different cases to show that the cascades between $p$ and $q$ are passing through $\mathfrak{M}_{2}(a,b)$. First of all, whenever $0\notin [\alpha,b]$, then by the definition of $\alpha$ (\ref{eqn:Aab}) there are no cascades passing through the non-compact component of ${\crit}(\cA^{H_{1}})$ and hence we can take $N=\emptyset$. If $0\in [\alpha,b)$, we set $N=K(b)$, with $K(b)$ defined as in (\ref{eqn:K(b)}), since all the cascades passing through the non-compact component of ${\crit}(\cA^{H_{1}})$ are in fact passing through $K(b)$. The last case we need to examine is $\cA^{H_{0}}(q)= b=0$: then we set $N=f_{1}^{-1}((-\infty,f_{1}(q)])$, since all the trajectories passing through the non-compact component of ${\crit}(\cA^{H_{1}})$ are confined in this case to $f_{1}^{-1}((-\infty,f_{1}(q)])$, which is also a compact subset of $\left(H_{1}\right)^{-1}(0)$ due to the coercivity of $f_{1}$. 

Let $\left\lbrace \left(\{u^{i}_{j}\}_{j=1}^{m_i},\{t^{i}_{j}\}_{j=1}^{m_i-1}\right)\right\rbrace_{i =1}^{k}$ be a perturbed broken trajectory with cascades between $p$ and $q$. This means the $u_j^i$'s are non-constant Floer trajectories, of which exactly one is a solution of $\partial_s u=\nabla_{J_s}\cA^{H_s}$, and the $t_j^i$'s represent the time necessary to flow along the positive gradient flow of $f_0$ or $f_1$, respectively, from one end of a Floer trajectory to the beginning of the next one (for more details, see \cite[Def. A5, Def. A.8]{Frauenfelder2004} and \cite[Def. 1.21]{Fauck2015}). Then for all $i=1,\dots k,\  j=1,\dots m$ 
$$
u^{i}_{j}\in \mathfrak{M}_{0}(a,b)\cup  \mathfrak{M}_{1}(a,b)\cup \mathfrak{M}_{2}(a,b),
$$
hence there exists a compact subset of $M\times \mathbb{R}$ that contains Floer components of all the broken trajectories with cascades between $p$ and $q$. In fact, by assumption (i) and (ii) the bounds are uniform and do not depend on the choice of $h\in \mathscr{O}(\Gamma)$.

Recall the evaluation maps (\ref{eqn:ev}) and define
\begin{equation}
K_{\Gamma(h)}(a,b):=\left( ev^{-}\left( \mathfrak{M}_{2}(a,b)\right)\cap \left( (H_{0})^{-1}(0) \times \{0\}\right)\right).
\label{eqn:KGab}
\end{equation}
It follows from boundedness of $\mathfrak{M}_{2}(a,b)$ in $M\times\mathbb{R}$, and a compactness argument as in Lemma \ref{compact_shade} that $K_{\Gamma(h)}(a,b)$ is a compact subset of $(H_{0})^{-1}(0)$. Uniform bounds on the Floer components together with the fact that both $f_{0}$ and $f_{1}$ are coercive, imply compactness of the Morse components of all the broken trajectories with cascades between $p$ and $q$.

Having obtained uniform bounds, we can now apply the usual Floer techniques to our setting: it follows from \cite[Thm. 11.1.6]{Audin2014} that for a generic choice of $h\in\mathscr{O}(\Gamma)$ and a pair $(p,q)\in {\crit}(f_{0})\times {\crit}(f_{1})$, all the moduli spaces ${\mathscr{M}}^{\Gamma(h)}(p,q)$ have the structure of a smooth manifold without boundary, of dimension $\mu(p)-\mu(q)$. That allows us to define a homomorphism between the chain complexes $CF(\cA^{H_{1}},f_{1})$ and $CF(\cA^{H_{0}},f_{0})$, by counting perturbed flow lines with cascades between points of equal indices. In order for this homomorphism to be well defined, one has to make sure that for all $q\in {\crit}(f_{1})$ and $a \in \mathbb{R}$
$$
\# \{ p \in {\crit}(f_{0})\ |\ \mu(p)=\mu(q),\ {\mathscr{M}}_{\Gamma(h)}(p,q)\neq \emptyset,\ \textrm{and}\ {\cA^{H}}(p) \geq a\} < + \infty.
$$
With $b=\cA^{H_{1}}(x)$, observe that ${\mathscr{M}}_{\Gamma(h)}(p,q)\neq \emptyset$ implies that $\cA^{H_{0}}(y)\in [a, \beta]$.

Recall that by Lemma \ref{lem:critVal} the number of connected components of ${\crit}(\cA^{H_{0}})\cap \left(\cA^{H_{0}}\right)^{-1}([a, \beta)\setminus \{0\})$ is finite and each of them is compact, hence
$$
\# \{ p \in {\crit}(f_{0})\ |\ \cA^{H_{0}}(p)\in [a, \beta]\setminus \{0\}\} < + \infty.
$$
On the other hand, if $\cA^{H_{0}}(y)=0$, then 
$$
y \in K(\beta) \cup f_{0}^{-1}((-\infty, \max_{K_{\Gamma(h)}(a,b)}f_{0}]),
$$
with the sets in the union above defined in (\ref{eqn:K(b)}) and (\ref{eqn:KGab}), respectively. The set $K(\beta)$ is compact by Lemma \ref{compact_shade}, as is the second set by virtue of a similar argument for $K_{\Gamma(h)}(a,b)$ and coercivity of $f_{0}$. As compact sets, they contain only finitely many critical points of $f_{0}$. This concludes the proof that the map between the chain complexes $CF(\cA^{H_{1}},f_{1})$ and $CF(\cA^{H_{0}},f_{0})$ is well defined. Moreover, for any pair $(p,q)\in {\crit}(f_{0})\times {\crit}(f_{1})$ the associated moduli space can be compactified (Theorem \ref{twr:compcas} and \cite[Thm. 11.1.6]{Audin2014}) by moduli spaces of lower index. This enables us to show that the associated map of chain complexes commutes with the differentials, hence it descends to the desired homomorphism $\Psi^{\Gamma(h)}$ on the homology level.
\end{proof}
Constructing the above homomorphism is only the first step in the proof of invariance. To prove that it is indeed an isomorphism one has to obtain similar bounds for the inverse homotopy $\Gamma^{-1}:=\{ H_{1-s},J_{1-s}\}_{s\in \mathbb{R}}$ and the concatenations
\begin{align}
\Gamma\#\Gamma^{-1} & := \{(H_{1-|2s-1|},J_{1-|2s-1|})\}_{s\in \mathbb{R}}\label{eqn:GG-1}\\
\Gamma^{-1}\#\Gamma & :=  \{(H_{|2s-1|},J_{|2s-1|})\}_{s\in \mathbb{R}}\label{eqn:G-1G}.
\end{align}
Note that if $\Gamma, \Gamma^{-1}, \Gamma\#\Gamma^{-1}$ and $\Gamma^{-1}\#\Gamma$ satisfy the assumptions of Lemma \ref{lem:homo}, then for a generic choice of $h$, $\Gamma(h), \Gamma(h)^{-1}, \Gamma(h)\#\Gamma(h)^{-1}$ and $\Gamma(h)^{-1}\#\Gamma(h)$ are regular (as intersections of generic sets are generic). In that case, by Lemma \ref{lem:homo}, the corresponding homomorphisms $\Psi^{\Gamma(h)}, \Psi^{\Gamma(h)^{-1}}$, $\Psi^{\Gamma(h)\#\Gamma(h)^{-1}}$ and $\Psi^{\Gamma(h)^{-1}\#\Gamma(h)}$ are well defined.

The next step would be to ensure that for a fixed generic $\bar{h}$ the $\lambda$-dependent families of homotopies $\{\overline{\Gamma}^{\lambda}(h)\}_{\lambda\in[0,1]}$ and $\{\widetilde{\Gamma}^{\lambda}(h)\}_{\lambda\in[0,1]}$, defined by 
\begin{align}
\overline{\Gamma}^{\lambda}(h) & :=\left\lbrace\left(\bar{H}_{\lambda(1-|2s-1|)}+\bar{h}_{\lambda(1-|2s-1|)}+h^{\lambda}_{s}, \bar{J}_{\lambda(1-|2s-1|)} \right)\right\rbrace_{s\in\mathbb{R}},\label{eqn:Gl(h)1}\\
\widetilde{\Gamma}^{\lambda}(h) & :=\left\lbrace\left(\bar{H}_{ |2\lambda s-1|}+\bar{h}_{|2\lambda s-1|}+h^{\lambda}_{s}, \bar{J}_{|2\lambda s-1|} \right)\right\rbrace_{s\in\mathbb{R}}, \label{eqn:Gl(h)2}
\end{align}
satisfy the assumptions of Lemma \ref{lem:homo} uniformly with respect to $\lambda \in [0,1]$ and $h\in \mathscr{O}^{\lambda}(\overline{\Gamma},\bar{h})\subseteq C^{\infty}_{c}([0,1]^{2}\times K)$. Observe that those homotopies of homotopies connect $\Gamma(\bar{h})\#\Gamma(\bar{h})^{-1}$ with the constant homotopy $(\bar{H}_{0},\bar{J}_{0})$ and $\Gamma(\bar{h})^{-1}\#\Gamma(\bar{h})$ with the constant homotopy $(\bar{H}_{1},\bar{J}_{1})$, respectively. If they do satisfy the assumptions of Lemma \ref{lem:homo} uniformly, then by \cite[Thm. 11.3.11, Prop. 11.2.8]{Audin2014} we can conclude that $\Psi^{\Gamma(\bar{h})\#\Gamma(\bar{h})^{-1}}$ and $\Psi^{\Gamma(\bar{h})^{-1}\#\Gamma(\bar{h})}$ are in fact isomorphisms on the homology level.

Finally, one would have to consider the $R$-dependent families of homotopies $\{\overline{\Gamma}^{R}(h)\}_{R\geq 0}$ and $\{\widetilde{\Gamma}^{R}(h)\}_{R\geq 0}$, defined by
\begin{align*}
\overline{\Gamma}^{R}(h) & := 
\left\lbrace\left(\bar{H}_{(1+e^{-R})(R+\frac{1}{2}-|s-\frac{1}{2}|)}+\bar{h}_{(1+e^{-R})(R+\frac{1}{2}-|s-\frac{1}{2}|)}+h^{R}_{s},\bar{J}_{(1+e^{-R})(R+\frac{1}{2}-|s-\frac{1}{2}|)}\right)\right\rbrace_{s\in\mathbb{R}}, \\
\widetilde{\Gamma}^{R}(h) & := \left\lbrace\left(\bar{H}_{(1+e^{-R})(|s-\frac{1}{2}|-R)}+\bar{h}_{(1+e^{-R})(|s-\frac{1}{2}|-R)}+h^{R}_{s}, \bar{J}_{(1+e^{-R})(|s-\frac{1}{2}|-R)}\right)\right\rbrace_{s\in\mathbb{R}},
\end{align*}
and check that they satisfy the assumptions of Lemma \ref{lem:homo} uniformly with respect to $R \geq 0$ and $h \in \mathscr{O}^{R}(\overline{\Gamma},\bar{h})\subseteq C_{c}^{\infty}([0,+\infty)\times[0,1]\times K)$. Observe that for $R=0,\ \overline{\Gamma}^{0}(h)= \Gamma(\bar{h})\#\Gamma(\bar{h})^{-1}$ and $\widetilde{\Gamma}^{0}(h)= \Gamma(\bar{h})^{-1}\#\Gamma(\bar{h})$. If $\{\overline{\Gamma}^{R}(h)\}_{R\geq 0}$ and $\{\widetilde{\Gamma}^{R}(h)\}_{R\geq 0}$ satisfy assumptions of Lemma \ref{lem:homo} uniformly, then it follows by \cite[Prop. 11.2.9]{Audin2014} that 
$$
\Psi^{\Gamma(\bar{h})^{-1}}\circ\Psi^{\Gamma(\bar{h})}=\Psi^{\Gamma(\bar{h})\#\Gamma(\bar{h})^{-1}}=\Psi^{\overline{\Gamma}^{0}(h)}\quad \textrm{and}\quad\Psi^{\Gamma(\bar{h})}\circ\Psi^{\Gamma(\bar{h})^{-1}}=\Psi^{\Gamma(\bar{h})^{-1}\#\Gamma(\bar{h})}=\Psi^{\widetilde{\Gamma}^{0}(h)}.
$$
Since constant homotopies induce the identity,
this shows that $\Psi^{\Gamma(\bar{h})}$ is an isomorphism between $RFH(H_{1},J_{1})$ and $RFH(H_{0},J_{0})$.
\section{Tentacular RFH}\label{sec:tent}

To prove Theorem \ref{thm:tentRFH}, we will show that the conditions in the definition of tentacular Hamiltonians guarantee the appropriate bounds for the Rabinowitz Floer equations based on the estimates in \cite{pasquotto2017}. We first verify that admissible Hamiltoians satisfy Hypotheses (H1)-(H3) in \cite{pasquotto2017}, and then we apply Theorem 1 therein to obtain the uniform bounds on the Floer trajectories which are necessary to construct the Rabinowitz Floer homology. Afterwards, we show that hypothesis \ref{h4} for strongly tentacular Hamiltonians implies uniform continuity of \ref{PO+}.
 This, combined with the result from Theorem \ref{thm:MB(K)}, ensures genericity of the Morse-Bott property in the affine space of compactly supported perturbations (under the assumption of  uniform continuity of \ref{PO+}), and allows us to construct the Rabinowitz Floer homology on a generic subset of strongly tentacular Hamiltonians. Finally, we describe how to obtain uniform estimates for homotopies of homotopies in order to prove invariance of Rabinowitz Floer homology with respect to different choices of the almost complex structure and compactly supported perturbations of the Hamiltonians. As a result, we will be able to extend the definition of Rabinowitz Floer homology to the whole set of strongly tentacular Hamiltonians.

\subsection{Strongly tentacular Hamiltonians}

For the reader's convenience, we recall hypotheses (H1)-(H3) from \cite{pasquotto2017}. 
In that paper, we consider a smooth Hamiltonian function $H$ on $(\mathbb{R}^{2n},\omega_{0})$ satisfying the following properties:
\begin{description}\label{def:H}
\item[\textbf{H1}\namedlabel{item:H1}{H1}] There exists a global Liouville vector field $X^{\dagger}$, and constants $c_{1},c_{2}>0$, $ c_{3}\geq 0$, such that for all $x\in \mathbb{R}^{2n}$ the following holds:
\[
|X^{\dagger}(x)| \leq  c_{1}(|x|+1), \quad \textrm{and} \quad dH(X^{\dagger})(x) \geq  c_{2}|x|^{2}-c_{3}.
\]
\item[\textbf{H2}\namedlabel{item:H2}{H2}] The function $H$ grows at most quadratically at infinity, that is, there exists a constant $L\geq 0$ such that
$$
\sup_{x\in \mathbb{R}^{2n}} \|D^{3}H_{x} \| \cdot |x| \leq  L <+\infty.
$$
\item[\textbf{H3}\namedlabel{item:H3}{H3}] There exist constants $c_{4},c_{5},\nu>0$ and a Liouville vector field $X^{\ddagger}$ defined on $H^{-1}((-\nu,\nu))$, such that
\begin{align*}
 |X^{\ddagger}(x)| & \leq c_{4}(|x|+1) \qquad \forall\ x\in H^{-1}((-\nu,\nu)),\\
& \inf_{H^{-1}((-\nu,\nu))}dH(X^{\ddagger}) \geq c_{5} >0.
\end{align*}
\end{description}

Observe that (H1) and (H2) are identical to \ref{h1} and \ref{h2} (under the additional assumption of asymptotic regularity of the Liouville vector field). In the following lemmas we will show that (H3) follows from \ref{h1}-\ref{h3}.

\begin{lemma}
\label{lem:uniform1}
For every admissible Hamiltonian $H$ there exist constants $\delta>0$ and $c_\delta>0$, and a vector field $X\in \scrL(\R^{2n})$, such that $dH(X)(x)\ge c_\delta$ for all $x\in H^{-1}(-\delta,\delta)$.
\end{lemma}

\begin{proof}
Let $X^\dagger$ and $X^\ddagger$ be the asymptotically regular Liouville vector fields whose existence is guaranteed by \ref{h1} and \ref{h3}, respectively.

Choose $0<\varepsilon<1$ and consider the asymptotically regular Liouville vector field $X_\varepsilon = (1-\varepsilon) X^\ddagger + \varepsilon X^\dagger$.
It follows from \ref{h1} that there exist constants $c,c'>0$, such that for all $x\in \R^{2n}$
\[ 
dH(X_\varepsilon)(x) = (1-\varepsilon) dH(X^\ddagger)(x) + \varepsilon dH(X^\dagger)(x)
 \ge (1-\varepsilon) dH(X^\ddagger)(x)- \varepsilon c' + \varepsilon c| x| ^{2}.
\]
By the contact type condition \ref{h3} we have that for every $r>0$ there exists $\varepsilon'>0$, such that $dH(X^\ddagger)(x)\ge \varepsilon'>0$ for all $x\in \Sigma \cap\{x\in \R^{2n}~:~| x| \le r\}$. Consequently, for small enough $\varepsilon$ there exists a constant $c''>0$, such that the above inequality becomes
\begin{equation}
dH(X_\varepsilon)(x)\ge c'' + c''| x| ^{2}, 
\label{dHXepsi}
\end{equation}
for all $x\in \Sigma$ --- a uniform contact type condition.
 
Asymptotic regularity of $X^\dagger$ and $X_\varepsilon$ and condition \ref{h2} imply that there exists a constant $a>0$, such that the following estimates holds for all $x\in \R^{2n}$:
\begin{equation}\label{constants}
 \|D X_\varepsilon(x)\| ,\| D^2 H(x)\| \le a \quad \textrm{and} \quad | X^\dagger(x)|, |X_\varepsilon(x)| , |\nabla H(x) |\le a + a |x|.
\end{equation}

Moreover, axiom \ref{h1} implies that for $|x|$ large enough
\begin{align}
|\nabla H(x)| |X^\dagger(x)| & \ge \langle \nabla H(x),X^\dagger(x)\rangle \ge c|x|^2-c',\nonumber \\
|\nabla H(x) |  & \ge \frac{c|x|^2-c'}{|X^\dagger(x)|} \ge \frac{c|x|^2-c'}{a+a|x|}\ge  \frac{c}{a}|x|-\frac{c+c'}{a}. \label{nablaH1}
\end{align}
In particular, we have that $|\nabla H(x)|\geq \varepsilon'>0$ outside of a sufficiently large ball. In fact, since $0$ is a regular value of $H$ (and hence $\nabla H$ does not vanish along $\Sigma=H^{-1}(0)$), we can conclude that for small enough $\varepsilon'>0$ there exists $\delta>0$, such that 
\begin{equation}
|\nabla H(x)|\geq \varepsilon'\quad \textrm{for all} \quad x\in H^{-1}((-\delta,\delta)).
\label{nablaH2}
\end{equation}
Combining \eqref{constants} with \eqref{nablaH1} we obtain the following estimate:
$$
 \frac{|X_\varepsilon(x)|}{|\nabla H(x)|}\leq \frac{(a|x|+a)^2}{c|x|^2-c'}.
$$
In particular, $\frac{|X_\varepsilon(x)|}{|\nabla H(x)|}<2\frac{a^2}{c}$ must hold for $|x|$ large enough, since the quotient converges to $\frac{a^2}{c}$ as $|x|\rightarrow +\infty$. This, together with \eqref{nablaH2} implies that there exists $a'>0$, such that
\begin{equation}
\frac{|X_\varepsilon(x)|}{|\nabla H(x)|}<a' \quad \textrm{for all} \quad x\in H^{-1}((-\delta,\delta)).
\label{XENH}
\end{equation}

We define the function $\sigma:H^{-1}((-\delta,\delta))\to\R$ by
\[
\sigma(x) :=dH(X_\varepsilon)(x)=\langle\nabla H(x),X_\varepsilon(x)\rangle.
\]
By the uniform contact type condition, \eqref{dHXepsi} for all $x\in \Sigma$ we have
\begin{equation}\label{sigma}
 \sigma(x)\ge   c''+c'' |x|^2\geq c''>0.
\end{equation}

Consider the normalized gradient flow defined by the equation $\frac{d}{ds}\phi^s(x) = \frac{\nabla H(x)}{| \nabla H(x)| ^2}$. 
This provides local coordinates $(x,s) \in \Sigma \times (-\delta,\delta)$,
via $H\circ \phi^s(x) = s$. In order to extend the estimate on $\sigma$ to a neighborhood of $\Sigma$, we estimate the $s$-derivative of $\sigma(\phi^s(x))$ :
\[
\begin{aligned}
\left| \frac{d}{ds} \sigma\circ \phi^s(x)\right|  &= \left| \left\langle D^2H\left(\frac{d}{ds}\phi^s(x)\right), X_\varepsilon (\phi^s(x)) \right\rangle + \left\langle \nabla H (\phi^s(x)),DX_\varepsilon  \left(\frac{d}{ds} \phi^s(x)\right) \right\rangle \right|  \\
& \le \|D^2H(\phi^s(x))\|\frac{|X_\varepsilon (\phi^s(x))|}{|\nabla H (\phi^s(x))|}+\|DX_\varepsilon (\phi^s(x))\|\le  a (a' +1)=:a'',
\end{aligned}
\]
uniformly for all $x\in \Sigma$. In the last line we have used the estimates from \eqref{constants} and \eqref{XENH}. It follows that for all $x\in \Sigma$ and all $s\in (-\delta, \delta)$ we have
\[
\sigma\circ\phi^s(x) \geq \sigma (x) - a''|s| \geq c''-a'' |s|,
\]
where the last inequality comes from \eqref{sigma}. Consequently, 
there exists a (possibly smaller) $\delta>0$ such that $\sigma(x) \ge \frac{c''}{2}>0$ for all $x\in H^{-1}((-\delta,\delta))$.

We denote the vector field $X_\varepsilon$ by $X$, which completes the proof.
\end{proof}

\begin{remark}
We have thus concluded that \ref{h1}-\ref{h3} imply (H3) in \cite{pasquotto2017}. Notice that if we were to consider only asymptotically regular vector fields, then the two sets of assumptions \ref{h1}-\ref{h3} and (H1)-(H3) would in fact be equivalent.
\end{remark}

In the following lemma we will show yet another property of admissible Hamiltonians: the length of the non-degenerate closed characteristics is linearly bounded by action.

\begin{lemma}
Let $H$ be an admissible Hamiltonian. Then there exist a constant $\bar{c}>0$ and an open subset $B(H)\subseteq C^{\infty}_{c}(\R^{2n})$, such that for all $h\in B(H)$ and all $(v,\eta)\in \crit(\cA^{H+h}),\ \eta\neq 0$,
$$
l(v)=\int |\partial_{t}v|dt\leq \bar{c} \left|\cA^{H+h}(v,\eta)\right|.
$$
\label{lem:lenght}
\end{lemma}
\begin{proof}
By axiom \ref{h1}, there exists a Liouville vector field $X^\dagger\in\scrL(\R^{2n})$, such that $dH(X^{\dagger})$ grows at least quadratically. Moreover, by \ref{h2} the Hamiltonian vector field $X_{H}$ grows at most linearly. Therefore, there exist $c, c'>0$, such that for all $x\in \R^{2n}$
\begin{equation}
|X_{H}(x)| \leq c \left( dH(X^{\dagger})(x)+c'\right).
\label{eqn:cond1}
\end{equation}
On the other hand, by Lemma \ref{lem:uniform1} there exists a Liouville vector field $X^{\ddagger}$, such that $dH(X^{\ddagger})$ is bounded away from $0$ in a neighborhood of $\Sigma=H^{-1}(0)$, therefore
\begin{align}
\inf_{\Sigma}dH(X^{\ddagger})\geq \frac{1}{c}.
\label{eqn:cond2}
\end{align}
Fixing $c$ to satisfy both \eqref{eqn:cond1} and \eqref{eqn:cond2} we define a subset of $C_{c}^{\infty}(\R^{2n})$ in the following way:
$$
B(H):= \left\lbrace\begin{array}{c|c}
 & \left|X_{H+h}(x)\right|< 2c ( d(H+h)(X^{\dagger})(x)+2c')\quad \forall\ x\in \mathbb{R}^{2n} \\
{\smash{\raisebox{.5\normalbaselineskip}{ $h\in C_{c}^{\infty}(\R^{2n})$}}} & d(H+h)(X^{\ddagger})(x)> \frac{1}{2c}\quad \forall\ x\in (H+h)^{-1}(0)
 \end{array}\right\rbrace
$$

Since both conditions defining $B(H)$ are open, $B(H)$ itself is open in $C_{c}^{\infty}(\R^{2n})$. Moreover, by \eqref{eqn:cond1} and \eqref{eqn:cond2} we can see that $0\in B(H)$ thus $B(H)$ is an open neighborhood of $0$ in $C_{c}^{\infty}(\R^{2n})$.
 
If we fix $h\in B(H)$ and take $(v,\eta)\in \crit(\cA^{H+h})$, with $\eta\neq 0$, we see that the period of the loop is bounded by the action:
$$
\left|\cA^{H+h}(v,\eta)\right|
=\left|\int \iota_{X^\ddagger}\omega(\partial_t v)-\eta\int (H+h)(v)\right|=\left| \int \omega (X^\ddagger,\eta X^{H+h})\right|=
|\eta|\int d(H+h)(X^{\ddagger})>\frac{|\eta|}{2c}.
$$
Let us now calculate the length of the loop $v$. Using the conditions defining $B(H)$ we get a uniform bound by the action:
\begin{align*}
l(v) & = \int |\partial_{t}v|=|\eta|\int \left|X_{H+h}(v)\right| \leq 2c |\eta|\int \left( d(H+h)(X^{\dagger})+2c'\right) \\
& \leq 2c(\left|\cA^{H+h}(v,\eta)\right| +2|\eta|c') < 2c \left|\cA^{H+h}(v,\eta)\right| (1+4 c' c)=\bar{c}\left|\cA^{H+h}(v,\eta)\right|,
\end{align*}
with $\bar{c}:=2c(1+4 c' c)$.
\end{proof}

In the following lemma we will show that the last property, \ref{h4} implies the uniform continuity of \ref{PO}. 

\begin{lemma}
Every tentacular Hamiltonian satisfies the axiom of uniform continuity of \ref{PO}.
\label{lem:PO}
\end{lemma}

\begin{proof}
We will start by constructing a sequence of compact sets $K_{a}\subseteq \mathbb{R}^{2n}$ and the corresponding sequence of open subsets $B(K_{a})\subseteq C_{c}^{\infty}(K_{a})$, and later we will show that they satisfy the hypothesis of uniform continuity of (PO).

Let $F:\mathcal{N}(\Sigma)\to\R$ be a function as in condition \ref{h4}, where $\mathcal{N}(\Sigma)$ is a neighborhood of $\Sigma=H^{-1}(0)$. Define
$$
A := \{x\in\mathcal{N}(\Sigma)\ |\ \{H,F\}\neq 0\},\qquad
B := \{x\in\mathcal{N}(\Sigma)\ |\ \{H,\{H,F\}\}>0\}.
$$
Then there exists $r>0$, such that $\Sigma\setminus B(r)\subseteq A\cup B$, where $B(r)$ denotes the open all of radius $r$ and centre in the origin.

Let $B(H)\subseteq C^{\infty}_{c}(\mathbb{R}^{2n})$ be the open neighborhood of $0$ from Lemma \ref{lem:lenght}.  We define a subset $\widetilde{B}(H)\subseteq B(H)$ in the following way: $h\in \widetilde{B}(H)$ if it satisfies the following conditions:
\begin{enumerate}[label=(\alph*)]
\item $(H+h)^{-1}(0)\setminus B(r)\subseteq A\cup B$;
\item for all $x\in (H+h)^{-1}(0)\setminus (B(r)\cup A)$
$$
|\{H,\{h,F\}\}(x)+\{h,\{H,F\}\}(x)+\{h,\{h,F\}\}(x)|<\{H,\{H,F\}\}(x);
$$
\item for all $x\in (H+h)^{-1}(0)\setminus (B(r)\cup B)$
$$
|\{h,F\}(x)|<|\{H,F\}(x)|.
$$
\end{enumerate}
First observe that $\widetilde{B}(H)\neq\emptyset$, since by axiom \ref{h4} we know that $0\in \widetilde{B}(H)$. Moreover, since the sets $A$ and $B$ are open and so are the above conditions, it follows that $\widetilde{B}(H)$ is an open neighborhood of $0$ in $C_{c}^{\infty}(\R^{2n})$. Now for every $a\in\mathbb{N}$, let $K_{a}$ be a closed ball of radius $r+2\bar{c}a$, where $\bar{c}$ is as in Lemma \ref{lem:lenght}
$$
K_{a}:=B_{0}(r+2\bar{c}a), \qquad B(K_{a}):=C_{c}^{\infty}(K_{a})\cap \widetilde{B}(H).
$$
Naturally, $\{K_{a}\}_{a\in\mathbb{N}}$ is an exhaustion of $\mathbb{R}^{2n}$ by compact sets. We will show that $\{B(K_{a})\}_{a\in\mathbb{N}}$ satisfies the hypothesis of uniform continuity of \ref{PO}. Fix $a\in\mathbb{N}$, $h\in B(K_{a})$, and take $(v,\eta)\in \crit(\cA^{H+h})$, with $0< |\cA^{H+h}(v,\eta)| \leq a$. By Lemma \ref{lem:lenght} we know that the length of the loop is uniformly bounded by the action $l(v)\leq \bar{c}a$. Let $\bar{v}:=\int v(t)dt$ be the average, then $|v(t)-\bar{v}|\leq l(v)\leq \bar{c} a$ and thus $v \subseteq B_{\bar{v}}(\bar{c} a)$. We will show now that $\bar{v}$ satisfies $|\bar{v}| \le \bar{c} a + r$. Suppose not, then it would follow that $|v(t)| \ge \left| |\bar{v}| - |v(t)-\bar{v}|\right| >r$, hence by \ref{a} we have 
\begin{equation}
v(S^{1})\subseteq (H+h)^{-1}(0)\setminus B(r)\subseteq A\cup B.
\label{S1subAB}
\end{equation}
The function $F\circ v$ obtains its maximum at the point $t_{0}\in S^{1}$, where 
$$
F\circ v(t_{0}) =\max_{S^{1}}F\circ v,\qquad
\frac{d}{dt}F\circ v (t_{0}) = 0 \qquad \textrm{and} \qquad \frac{d^{2}}{dt^{2}}F\circ v (t_{0})\leq 0.
$$
Since $(v,\eta)\in \crit\left(\mathcal{A}^{H+h}\right)$, we have
\begin{align}
0 & =\frac{d}{dt}F\circ v (t_{0}) = d F(\partial_{t}v(t_{0})) =\eta dF\left(X_{H+h}(v(t_{0})\right)= \eta \{H+h,F\}(v(t_{0})),\label{dtFv}\\
0 & \geq \frac{d^{2}}{dt^{2}}F\circ v (t_{0}) = \eta^{2} d\left(\{H+h,F\}\right)(X_{H+h}(v(t_{0}))=\eta^{2}\{H+h,\{H+h,F\}\}(v(t_{0})).\label{dt2Fv}
\end{align}
By \eqref{S1subAB} we know that $v(t_{0})\in A\cup B$. Therefore, we can consider the two following cases:
\begin{enumerate}
\item $v(t_{0})\notin A$: then $v(t_{0})\in (H+h)^{-1}(0)\setminus (B(r)\cup A)$ and we have that \ref{b} contradicts \eqref{dt2Fv} in the following way:
\begin{align*}
0 & \geq \frac{1}{\eta^2}\frac{d^{2}}{dt^{2}}F\circ v (t_{0}) = \{H+h,\{H+h,F\}\}(v(t_{0}))\\
 & = \{H,\{h,F\}\}(v(t_{0}))+\{h,\{H,F\}\}(v(t_{0})) +\{h,\{h,F\}\}(v(t_{0})) +\{H,\{H,F\}\}(v(t_{0})) >0.
\end{align*}
\item $v(t_{0})\notin B$: then $v(t_{0})\in (H+h)^{-1}(0)\setminus (B(r)\cup B)$ and we have that \ref{c}  contradicts with \eqref{dtFv}:
$$
0 =\left|\frac{1}{\eta}\frac{d}{dt}F\circ v (t_{0})\right| = \left|\{H+h,F\}(v(t_{0}))\right| \geq |\{H,F\}(v(t_{0}))|-|\{h,F\}(v(t_{0}))|>0.
$$
\end{enumerate}
Summarizing, we can conclude that for $h\in B(K_{a})$, if $(v,\eta)$ is a non-degenerate critical point of $\cA^{H+h}$ with action $0<|\cA^{H+h}(v,\eta)|\leq a$, then $v(S^{1})\subseteq K_{a}$. In other words, \ref{PO} is uniformly continuous at $H$.
\end{proof}

However, in order to construct Rabinowitz Floer homology, we need a stronger condition, namely uniform continuity of \ref{PO+}. In the following lemma we will show that the choice of a coercive function in Axiom \ref{h4} ensures uniform continuity of \ref{PO+}. As a result, \ref{PO+} is uniformly continuous for the strongly tentacular Hamiltonians. In fact we can state the result in greater generality, for a Hamiltonian on any exact symplectic manifold, not necessarily on $\R^{2n}$.

\begin{lemma}
Let $H$ and $F$ be smooth Hamiltonians on a symplectic manifold $(M,\omega)$. If $F$ is coercive and the set
$$
K_{0} :=\left\lbrace x\in H^{-1}(0)\  \Big|\  \{H,\{H,F\}\}(x) \leq 0\quad \textrm{and}\quad \{H,F\}(x)= 0\right\rbrace, 
$$
is compact, then \ref{PO+} is uniformly continuous at $H$.
 \label{lem:PO+}
\end{lemma}
\begin{proof}
For $n=0, 1, 2, \dots$ let us define
$$
K_{n} := F^{-1}((-\infty,n+\sup_{K_{0}}F]).
$$
Since $F$ is coercive, every $K_{n}$ is compact and we obtain an exhaustion of $M$ by compact sets
$$
M = \bigcup_{n\in \mathbb{N}} K_{n}.
$$
We will show that the family $\{K_{n}\}_{n\in\mathbb{N}}$ satisfies the hypothesis of uniform continuity of \ref{PO+}. More precisely, we will show that for every $n=1, 2, \dots$, and every $h \in C^{\infty}_{c}(K_{n})$  
$$
(v,\eta)\in \crit(\cA^{H+h})\setminus \left(\cA^{H+h}\right)^{-1}(0) \qquad \textrm{implies} \qquad v(t) \in K_{n}\quad \forall\ t\in S^{1}.
$$
Fix $h \in C^{\infty}_{c}(K_{n})$, take $(v,\eta) \in \crit(\mathcal{A}^{H+h})\setminus \left(\mathcal{A}^{H+h}\right)^{-1}(0)$ and calculate the maximum of $F$ over $v$. There exists a $t_{0}\in S^{1}$, such that
$$
F\circ v(t_{0}) =\max_{S^{1}}F\circ v,\qquad
\frac{d}{dt}F\circ v (t_{0}) = 0 \qquad \textrm{and} \qquad \frac{d^{2}}{dt^{2}}F\circ v (t_{0})\leq 0.
$$
Now suppose that $v(S^{1})\setminus K_{n} \neq \emptyset$. Then $v(t_{0})\notin K_{n}$, so in particular it is also not in the support of $h$. Since $(v,\eta)\in \crit\left(\mathcal{A}^{H+h}\right)$, we would have
\begin{align*}
0 & =\frac{d}{dt}F\circ v (t_{0}) = d F(\partial_{t}v(t_{0})) =\eta dF\left(X_{H}(v(t_{0})\right)= \eta \{H,F\}(v(t_{0})),\\
0 & \geq \frac{d^{2}}{dt^{2}}F\circ v (t_{0}) = \eta^{2} d\left(\{H,F\}\right)(X_{H}(v(t_{0}))=\eta^{2}\{H,\{H,F\}\}(v(t_{0})) .
\end{align*}
But this would imply that $v(t_{0})\in K_{0}\subseteq K_{n}$, which leads us to contradiction and concludes the proof.
\end{proof}

\subsection{Estimates for Floer trajectories} In this subsection we will recall Theorem 1. from \cite{pasquotto2017}, which guarantees uniform estimates for tentacular Hamiltonians.

Let us consider the following setting: let $J_{0}$ be the standard almost complex structure on $\mathbb{R}^{2n}$. Fix an admissible Hamiltonian $H:\mathbb{R}^{2n}\to \mathbb{R}$, a non-empty compact set $K\subseteq \mathbb{R}^{2n}$, an open, precompact subset $\mathcal{V} \subseteq \mathbb{R}^{2n}$, and a constant $\mathfrak{y}>0$. Let $\mathscr{O}(H)\subseteq C_{c}^{\infty}(K)$ be the open subset associated to $H$ and $K$ by \cite[Lem. 2.1]{pasquotto2017} and let $\tilde{c},\varepsilon_{0}>0$ be the constants associated to $H$ and $K$ by \cite[Lem. 3.1]{pasquotto2017}.
\begin{theorem}
\label{twr:ModuliCompact}
Let $\Gamma:=\{(H_s,J_s)\}_{s\in\R}$ be a homotopy as described in \eqref{homotopy1} with
\begin{align}
H_s\in H+\mathscr{O}(H),\quad J_s \in \mathscr{J}^{\infty}\big(\mathbb{R}^{2n},\omega_{0},\mathcal{V} & \times \left((-\infty,-\mathfrak{y})\cup( \mathfrak{y}, \infty)\right)\big), \nonumber \\
\left(\tilde{c}+\tfrac{1}{\varepsilon_{0}}\|J\|_{L^{\infty}}^{\frac{3}{2}}\right)\|\partial_{s}H_{s}\|_{L^{\infty}} & < \frac{1}{8}. \label{eqn:Hs}
\end{align}
If we fix $a,b\in \mathbb{R}$ and a compact subset $N\subseteq H^{-1}_{1}(0)$, then for each pair $(\Lambda_{0},\Lambda_{1})$ of connected components
$$
\Lambda_{0} \subseteq \crit(\cA^{H_{0}})\cap (\cA^{H_{0}})^{-1}([a,\infty)),\qquad
\Lambda_{1}  \subseteq \mathscr{C}(\cA^{H_{1}},N) \cap (\cA^{H_{1}})^{-1}((-\infty,b]),
$$
the image in $\mathbb{R}^{2n+1}$ of the moduli space $\mathscr{M}^{\Gamma}(\Lambda_{0},\Lambda_{1})$ is uniformly bounded.
\end{theorem}
Note that the theorem stated above differs slightly from \cite[Thm. 1]{pasquotto2017}, since we consider $N$ to be a compact subset of $H^{-1}_{1}(0)$ rather than $H^{-1}_{0}(0)$. This change is due to the fact that we dropped the assumption of bounded topology of the hypersurface, and thus need to require the Morse function used in the construction of the cascades to have compact sublevel sets rather than compact superlevel sets as in \cite{Wisniewska2017}. Yet, the change does not influence the arguments in the proof - we just assume that the Floer trajectories end in a compact subset of $\mathbb{R}^{2n}$ rather than start in it. Thus the theorem continues to hold also in this setting.

\subsection{RFH for tentacular Hamiltonians}
The goal of this subsection is to prove Theorem \ref{thm:tentRFH}. We do this by first proving, in Lemma \ref{lem:RFHdense}, that Rabinowitz Floer homology is well defined on a dense subset of pairs of strongly tentacular Hamiltonians and almost complex structures. Secondly, in Lemma \ref{lem:invJ} we show that the Rabinowitz Floer homology on this set is in fact independent of the choice of the almost complex structure. Finally, in Lemma \ref{lem:invH} we show that the Rabinowitz Floer homology is invariant under small enough compactly supported perturbations of the Hamiltonians, allowing us to define it on the whole subset of strongly tentacular Hamiltonians.

Let $J_{0}$ be the standard almost complex structure on $\mathbb{R}^{2n}$. Define
\begin{align*}
\mathscr{J}^{\mathcal{V}}_{\mathfrak{y}} & :=\mathscr{J}^{\infty}\left(\mathbb{R}^{2n},\omega_{0},\mathcal{V}\times \left((-\infty,-\mathfrak{y})\cup(\mathfrak{y},\infty)\right)\right),\\
\mathscr{J}_{\bigstar}& := \bigcup_{\mathfrak{y}>0}\bigcup_{\mathcal{V}\subseteq \mathbb{R}^{2n}} \mathscr{J}_{\eta}^{\mathcal{V}},
\end{align*}
where the union is taken over all $\mathfrak{y}>0$ and all open, pre-compact subsets $\mathcal{V}\subseteq \mathbb{R}^{2n}$. The space $\mathscr{J}_{\bigstar}$ is a contractible subspace of $\mathscr{J}^{\infty}(\mathbb{R}^{2n},\omega_{0})$ with the limit topology. Recall that $\mathscr{H}(\mathbb{R}^{2n})$ denotes the set of strongly tentacular Hamiltonians. 

\begin{lemma}
There exists a dense subset of $\mathscr{H}(\mathbb{R}^{2n})\times \mathscr{J}_{\bigstar}$ on which Rabinowitz Floer homology is well defined.
\label{lem:RFHdense}
\end{lemma}
\begin{proof}
First observe that $\mathscr{H}(\mathbb{R}^{2n})$ is open under compact perturbations. In other words for every $H\in \mathscr{H}(\mathbb{R}^{2n})$ there exists an open subset $B(H)\subseteq C_{c}^{\infty}(\mathbb{R}^{2n})$, such that $H+B(H)\subseteq \mathscr{H}(\mathbb{R}^{2n})$. By Lemma \ref{lem:PO+} we know that \ref{PO+} is uniformly continuous at $H$, hence by Theorem \ref{thm:MB(K)} there exists a generic subset $A(H)\subseteq B(H)$, such that for every $h\in A(H)$ the functional $\cA^{H+h}$ satisfies \ref{MB} and its set of critical values is discrete. Define
$$
\mathscr{H}^{reg}(\mathbb{R}^{2n}):=\bigcup_{H\in \mathscr{H}(\mathbb{R}^{2n})}(H+A(H)),
$$
and fix $H\in \mathscr{H}^{reg}(\mathbb{R}^{2n})$. By uniform continuity of \ref{PO+} at $H$ all the nondegenerate periodic orbits of $X_{H}$ on $\Sigma$ are uniformly bounded. In particular, there exists an open, precompact set $\mathcal{V} \subseteq \mathbb{R}^{2n}$ containing all of them. By Lemma \ref{lem:etaInf} the infimum of the $\eta$ parameter over $\crit(\mathcal{A}^{H})\setminus \left(\Sigma_0\right)$ is bounded away from $0$, that is, for a sufficiently small $\mathfrak{y}>0$ all the non-degenerate periodic orbits are contained in $\mathcal{V} \times \left((-\infty,-\mathfrak{y})\cup( \mathfrak{y}, \infty)\right)$. On the other hand, we can apply Theorem \ref{twr:ModuliCompact} to a ``constant'' homotopy $(H_s,J_s)\equiv (H,J)$ with $J\in \mathscr{J}^{\mathcal{V}}_{\mathfrak{y}}$, and obtain uniform bounds on the moduli spaces $\mathscr{M}_{J,H}(\Lambda^{-},\Lambda^{+})$ for any pair of connected components $\Lambda^{-},\Lambda^{+}\subseteq \crit(\cA^{H})$. In particular, the estimates on the moduli spaces depend continuously on the choice of almost complex structure. We can thus apply Theorem \ref{twr:trasvers} and Theorem \ref{thm:defRFH} to infer that for a generic subset of $\mathscr{J}^{\mathcal{V}}_{\mathfrak{y}}$ the corresponding $RFH(H,J)$ is well defined. Since the result holds for any big enough $\mathcal{V}\subseteq \mathbb{R}^{2n}$ and small enough $\mathfrak{y}>0$, we can conclude that the Rabinowitz Floer homology is well defined for a generic subset of $\mathscr{H}(\mathbb{R}^{2n})\times \mathscr{J}_{\bigstar}$.
\end{proof}

\begin{lemma}
For any $H\in \mathscr{H}^{reg}(\mathbb{R}^{2n})$ the associated Rabinowitz Floer homology is independent of the choice of almost complex structure.
\label{lem:invJ}
\end{lemma}

\begin{proof}
Fix $H\in \mathscr{H}^{reg}(\mathbb{R}^{2n})$ and let ${J}_{0},{J}_{1}\in \mathscr{J}_{\bigstar}$ be two almost complex structures for which $RFH(H,J_{0})$ and $RFH(H,J_{1})$ are well defined. There exists a homotopy $\overline{\Gamma}:=\{(H,J_{s})\}_{s\in\mathbb{R}}$ between $(H,J_{0})$ and $(H,J_{1})$ satisfying \eqref{homotopy1} and with $J_{s} \in \mathscr{J}^{\mathcal{V}}_{\mathfrak{y}}$ for some small enough $\mathfrak{y}>0$ and big enough, open, pre-compact set $\mathcal{V}\subseteq \mathbb{R}^{2n}$. Fix a compact set $K\subseteq \mathbb{R}^{2n},\ K\neq \emptyset$, and let $\mathscr{O}(H)\subseteq C_{c}^{\infty}(K)$ be the open subset associated to $H$ and $K$ by \cite[Lem. 2.1]{pasquotto2017}, and $\tilde{c},\varepsilon_{0}>0$ the constants associated to $H$ and $K$ by \cite[Lem. 3.1]{pasquotto2017}. Define
$$
\mathscr{O}({J}_{1},{J}_{2}):=\left\lbrace\begin{array}{c|c}
 & \|\partial_{s}h\|_{\infty}<\frac{1}{16}\left(\tilde{c}+\frac{1}{\varepsilon_{0}}\|J_{s}\|_{\infty}^{\frac{3}{2}}\right)^{-1}\\
{\smash{\raisebox{.5\normalbaselineskip}{ $h\in C^{\infty}_{0}([0,1]\times K)$}}} & \forall\ s\in [0,1]\quad h_{s} \in \mathscr{O}(H)
 \end{array}\right\rbrace
$$
Then $\mathscr{O}({J}_{1},{J}_{2})$ is a neighborhood of $0$ in $C^{\infty}_{c}([0,1]\times K)$. Moreover, for every $h\in \mathscr{O}({J}_{1},{J}_{2})$, the homotopy 
$$
\Gamma(h):=\{(H+h_{s},J_{s})\}_{s\in\mathbb{R}}
$$
satisfies inequality (\ref{eqn:Hs}). Moreover, the Hamiltonian part of the homotopy lies in $H+\mathscr{O}(H)$. In other words, for every $h\in \mathscr{O}({J}_{1},{J}_{2})$, $\Gamma(h)$ satisfies the assumptions of Theorem \ref{twr:ModuliCompact} and as a result one obtains uniform $L^{\infty}\times\mathbb{R}$ bounds for the corresponding moduli spaces and the energy as shown in \cite[Prop. 3.3]{pasquotto2017}. Moreover, using the linear dependence between the action and the $\eta$ parameter (\cite[Lem. 2.1]{pasquotto2017}), together with inequality (\ref{eqn:Hs}), we can directly apply \cite[Cor. 3.8]{CieliebakFrauenfelder2009} to show that $\Gamma(h)$ satisfies the Novikov conditions. In fact, the estimates do not depend on the chosen perturbation, but can be obtained uniformly over the whole set $\mathscr{O}({J}_{1},{J}_{2})$. Consequently, the homotopy $\{H,J_{s}\}_{s\in\mathbb{R}}$ and the set $\mathscr{O}({J}_{1},{J}_{2})$ satisfy the assumptions of Lemma \ref{lem:homo} and we can deduce that for a generic choice of $h\in \mathscr{O}({J}_{1},{J}_{2})$ there exists a homomorphism
$$
\Psi^{\Gamma(h)}:RFH(H,{J}_{2})\to RFH(H,{J}_{1}).
$$
Note that for every $h\in \mathscr{O}({J}_{1},{J}_{2})$, the homotopies $\Gamma(h)^{-1}$, $\Gamma(h)\#\Gamma(h)^{-1}$ and $\Gamma(h)^{-1}\#\Gamma(h)$ (cf. \eqref{eqn:GG-1} and \eqref{eqn:G-1G}) also satisfy inequality (\ref{eqn:Hs}), and their Hamiltonian components lie in $H+\mathscr{O}(H)$. By the same argument as the one above, we can conclude that for a generic choice of $h\in \mathscr{O}({J}_{1},{J}_{2})$ the four homomorphisms $\Psi^{\Gamma(h)}, \Psi^{\Gamma(h)^{-1}}$, $\Psi^{\Gamma(h)\#\Gamma(h)^{-1}}$ and $\Psi^{\Gamma(h)^{-1}\#\Gamma(h)}$ are well defined.
The set $\mathscr{O}({J}_{1},{J}_{2})$ is convex and open, hence for every fixed, generic $\bar{h}\in \mathscr{O}({J}_{1},{J}_{2})$ the set
$$
\mathscr{O}^{\lambda}(\overline{\Gamma},\bar{h}):=\left\lbrace h\in C^{\infty}_{c}([0,1]^{2}\times K)\ \Big| \ 2\lambda\bar{h}+ h^{\lambda}\in 2\mathscr{O}({J}_{1},{J}_{2})\right\rbrace
$$
is an open subset of $C^{\infty}_{c}([0,1]^{2}\times K)$. Moreover, for every $h\in \mathscr{O}^{\lambda}(\overline{\Gamma},\bar{h})$ the corresponding homotopies $\overline{\Gamma}^{\lambda}(h)$ and $\widetilde{\Gamma}^{\lambda}(h)$ as in \eqref{eqn:Gl(h)1} and \eqref{eqn:Gl(h)2} also satisfy inequality (\ref{eqn:Hs}) and their Hamiltonian parts lie in $H+\mathscr{O}(H)$. Therefore, the corresponding moduli spaces are uniformly bounded by Theorem \ref{twr:ModuliCompact} and the homotopies satisfy assumptions of Lemma \ref{lem:homo} uniformly. This proves that $\Psi^{\Gamma(\bar{h})\#\Gamma(\bar{h})^{-1}}$ and $\Psi^{\Gamma(\bar{h})^{-1}\#\Gamma(\bar{h})}$ are isomorphisms on the homology level. Using similar arguments one can prove uniform bounds for the $R$-parametric families of homotopies, which ensures that the homomorphism associated to the concatenations $\Gamma(\bar{h})\#\Gamma(\bar{h})^{-1}$ and $\Gamma(\bar{h})^{-1}\#\Gamma(\bar{h})$ are just the composition of the homomorphisms $\Psi^{\Gamma(\bar{h})}$ and $\Psi^{\Gamma(\bar{h})^{-1}}$. This proves that $\Psi^{\Gamma(\bar{h})}$ is an isomorphism and hence concludes the proof.
\end{proof}
\begin{remark}
In view of Lemma \ref{lem:invJ}, from now on we will omit the almost complex structure from the notation and write $RFH(H)$ for the Rabinowitz Floer homology of $H\in\scrH^{reg}(\R^{2n})$.
\end{remark}
\begin{lemma}
For any $H\in \mathscr{H}(\mathbb{R}^{2n})$ there exists an open neighborhood $\widetilde{\mathscr{O}}(H)\subseteq C_{c}^{\infty}(\mathbb{R}^{2n})$, such that for every pair $h_{0}, h_{1} \in \widetilde{\mathscr{O}}(H)$ with $H+ h_{0}, H+h_{1} \in \mathscr{H}^{reg}(\mathbb{R}^{2n})$
$$
RFH(H+ h_{0}) \cong RFH( H+h_{1}).
$$
\label{lem:invH}
\end{lemma}
\begin{proof}
Let $B(H)\subseteq C^{\infty}_{c}(\R^{2n})$ and $\{K_{m}\}_{m\in\N}$ be the open subspace and the sequence of compact subsets of $\R^{2n}$, respectively, whose existence is guaranteed by the uniform continuity of \ref{PO+} at $H$. Fix one of the compact subsets $K:=K_{m}$. Let $\mathscr{O}(H)\subseteq C_{c}^{\infty}(K)$ be the convex, open subset from \cite[Lem. 2.1]{pasquotto2017}, and let $\tilde{c},\varepsilon_{0}>0$ be the constants associated to $H$ and $K$ by \cite[Lem. 3.1]{pasquotto2017}. Define
$$
\widetilde{\mathscr{O}}(H):= \left\lbrace \begin{array}{c | c}
h \in \mathscr{O}(H)\cap B(H)
& \|h\|_{\infty}<2^{-7}\left(\tilde{c}+\frac{\sqrt{8}}{\varepsilon_{0}}\right)^{-1}
\end{array}\right\rbrace.
$$
Fix $h_{0}, h_{1} \in \widetilde{\mathscr{O}}(H)$. By uniform continuity of \ref{PO+} at $H$ and Lemma \ref{lem:etaInf} there exists an open, precompact subset $\mathcal{V} \subseteq \R^{2n}$ and a constant $\mathfrak{y}>0$, such that
all nondegenerate critical points of $\cA^{H+ h_{0}}$ and  $\cA^{H+h_{1}}$ lie in $\mathcal{V} \times \left((-\infty,-\mathfrak{y})\cup( \mathfrak{y}, \infty)\right)$. By Theorem \ref{twr:trasvers} the sets of regular almost complex structures corresponding to $H+h_{0}$ and $H+h_{1}$, respectively, are dense in $\mathscr{J}^{\mathcal{V}}_{\mathfrak{y}}$. Moreover, this set is contractible. Therefore there exist two almost complex structures ${J}_{0}, {J}_{1}\in \mathscr{J}^{\mathcal{V}}_{\mathfrak{y}}$, regular in the sense of Theorem \ref{twr:trasvers} for $H+ h_{0}$ and $H+ h_{1}$ respectively, and a homotopy $\{J_{s}\}_{s\in\R}$ connecting them, with $J_{s} \in \mathscr{J}^{\mathcal{V}}_{\mathfrak{y}}$ for all $s\in\mathbb{R}$, $\|J_{s}\|_{\infty} < 2$,
$$
J_{s} ={J}_{1}\ \textrm{for}\ s\leq 0\quad \textrm{and}\quad J_{s} ={J}_{2}\ \textrm{for}\ s\geq 1.
$$
If we choose a smooth function $\beta \in C^{\infty}(\R)$ satisfying  $\|\beta '\|_{\infty} < 2$, and 
$$
\beta(s)=0,\ \textrm{for}\ s\leq 0\quad \textrm{and}\quad \beta(s)=1\ \textrm{for}\ s\geq 1,
$$
the homotopy
$$
\overline{\Gamma}:=\left\lbrace (H+h_{0}(1-\beta(s))+h_{1}\beta(s)),J_{s}\right\rbrace_{s\in \R}
$$
connects $(H+h_{0},{J}_{0})$ with $(H+h_{1},{J}_{1})$. Moreover, $\Gamma$ satisfies \eqref{eqn:Hs} and due to convexity of $\mathscr{O}(H)$ the Hamiltonian part of $\overline{\Gamma}$ lies in $H+\mathscr{O}(H)$. This puts us in the setting of Theorem \ref{twr:ModuliCompact}, which gives us uniform $L^{\infty}\times\mathbb{R}$ bounds on the corresponding moduli spaces and uniform bounds on the energy as shown in \cite[Prop. 3.3]{pasquotto2017}. Using arguments similar to those in the proof of Lemma \ref{lem:invJ}, we can show that the uniform bounds hold also for perturbations of $\overline{\Gamma}$, their inverses and concatenations and homotopies of homotopies. Having established uniform bounds on the families of homotopies, one can use standard, Floer-theoretic techniques (as presented in section \ref{sec:InvRFH}) to conclude that $RFH(H+ h_{0})$ and $RFH( H+h_{1})$ are indeed isomorphic.
\end{proof}

We conclude this section with the proof of the main theorem:\\

\noindent\textit{Proof of Theorem \ref{thm:tentRFH}}: By Lemma \ref{lem:RFHdense} and Lemma \ref{lem:invJ} there exists a dense subset of the strongly tentacular Hamiltonians for which the Rabinowitz Floer homology is well defined. We denote it by $\mathscr{H}^{reg}(\R^{2n})$. Let $H$ be a strongly tentacular Hamiltonian and let $\widetilde{O}(H)\subseteq C_{c}^{\infty}(\R^{2n})$ be the open neighborhood of $0$ defined in Lemma \ref{lem:invH}. For $h\in \widetilde{O}(H)$ and $H+h\in \mathscr{H}^{reg}(\R^{2n})$ we define
$$
RFH(H):=RFH(H+h). 
$$
This definition is independent on the choice of the perturbation $h\in \widetilde{O}(H)$, since by Lemma \ref{lem:invH} for any pair $h_{1},h_{2}\in \widetilde{O}(H)$, such that $H+h_{1},H+h_{2}\in \mathscr{H}^{reg}(\R^{2n})$ the corresponding Rabinowitz Floer homologies are isomorphic.

\hfill $\square$
\begin{remark}
The proof of the above theorem also implies that for any one-parameter family of tentacular Hamiltonians $\{H_s\}$ in the affine space of compactly supported perturbations of a given Hamiltonian $H$, the Rabinowitz Floer homology is constant along $\{H_s\}$.
\end{remark}

\subsection{Invariance under symplectomorphisms}
\label{sympmor}

The goal of this section is to prove Corollary \ref{cor:SympTentRFH} and to determine which subgroup of the group of symplectomorphisms leaves the set of strongly tentacular Hamiltonians invariant.
\SympTentRFH*
\begin{proof}
Let $\varphi$ be a diffeomorphism of $\R^{2n}$. In particular, $\varphi$ takes compact sets to compact sets. We define a map
$$
\tilde{\varphi}: C^{\infty}_{c}(\R^{2n}) \to C_c^{\infty}(\R^{2n}),\qquad
f \mapsto f \circ \varphi^{-1}.
$$
The map $\tilde{\varphi}$ is a diffeomorphism of $ C_c^{\infty}(\R^{2n})$, so it maps open sets to open sets and dense sets to dense sets.

Let $H$ be a strongly tentacular Hamiltonian and let $\varphi$ be a symplectomorphism of $\R^{2n}$. Let $\widetilde{\scrO}(H)\subseteq C_{c}^{\infty}(\R^{2n})$ be the open neighborhood of $0$ associated to $H$ by Lemma \ref{lem:invH}. Then $\widetilde{\scrO}(H \circ \varphi^{-1}):=\tilde{\varphi}\left( \widetilde{\scrO}(H) \right) $ is an open subset of $0$ in $C_c^{\infty}(\R^{2n})$. On the other hand, by Lemma \ref{lem:RFHdense} and Lemma \ref{lem:invJ}, there exists a dense subset of the set of strongly tentacular Hamiltonians, denoted by $\mathscr{H}^{reg}(\R^{2n})$, for which the Rabinowitz Floer homology is well defined. Define
$$
\widetilde{\scrO}^{reg}(H \circ \varphi^{-1}):=\left\lbrace h \circ \varphi^{-1} \, \Big|\, h \in \widetilde{\scrO}(H), \, H+h \in \mathscr{H}^{reg}(\R^{2n}) \right\rbrace.
$$
In view of Lemma \ref{lem:RFHdense} and Lemma \ref{lem:invJ} the set $\tilde{\varphi}^{-1}\left(\widetilde{\scrO}^{reg}(H \circ \varphi^{-1})\right)$ is dense in $\widetilde{\scrO}(H)$. Consequently, the set $\widetilde{\scrO}^{reg}(H \circ \varphi^{-1})$ is dense in $\widetilde{\scrO}(H \circ \varphi^{-1})$. By Proposition \ref{prop:sympRFHinv}, for every $h \in \widetilde{\scrO}^{reg}(H \circ \varphi^{-1})$ the corresponding Rabinowitz Floer homology $RFH_*(H \circ \varphi^{-1}+h)$ is well defined. What is left to show is that for all $h\in \widetilde{\scrO}^{reg}(H \circ \varphi^{-1})$ the corresponding $RFH(H \circ \varphi^{-1} + h)$ is independent of the choice of $h$. Let us take $h_1,h_2 \in \widetilde{\scrO}^{reg}(H \circ \varphi^{-1})$. Then we have the following sequence of isomorphisms:
$$
RFH(H\circ \varphi^{-1}+h_1) \cong RFH(H+h_1\circ \varphi) \cong RFH(H+ h_2\circ \varphi)\cong RFH(H\circ \varphi^{-1}+ h_2),
$$
where the first and last isomorphisms follow from Proposition \ref{prop:sympRFHinv} and the middle isomorphism is a corollary of Lemma \ref{lem:invH}. As a result, we can define
$$
RFH(H\circ \varphi^{-1}):=RFH(H\circ \varphi^{-1}+h),
$$
where $h\in  \widetilde{\scrO}^{reg}(H \circ \varphi^{-1})$.
\end{proof}

Even though the Rabinowitz Floer homology of the image of a strongly tentacular Hamiltonian under a symplectomorphism is well-defined, this image will not, in general, be a strongly tentacular Hamiltonian. The next lemma describes which elements of $\operatorname{Symp}(\R^{2n})$ preserve the set of tentacular Hamiltonians.
\begin{lemma}
The set of (strongly) tentacular Hamiltonians is invariant under the action of the group
$$
\operatorname{Symp}^{ten}(\R^{2n}):=\left\lbrace \begin{array}{c|c c }
& \sup\|D^{k}\varphi(x)\|\cdot |x|^{k-1} <+ \infty, & \\
{\smash{\raisebox{.5\normalbaselineskip}{$\varphi \in \operatorname{Symp}(\R^{2n})$}}} & \sup\|D^{k}\varphi^{-1}(x)\|\cdot |x|^{k-1}<+\infty, & {\smash{\raisebox{.5\normalbaselineskip}{ $k=1,2, 3$}}}
\end{array}\right\rbrace.
$$
\label{lem:SympTen}
\end{lemma}
\begin{proof}
We will show that for every (strongly) tentacular Hamiltonian $H$ and any $\varphi\in \operatorname{Symp}^{ten}(\R^{2n})$, the composition $H\circ \varphi$ is also a (strongly) tentacular Hamiltonian. Since $\varphi$ is a symplectomorphism it preserves the symplectic properties. On the other hand, the bounds on the derivatives of $\varphi$ and $\varphi^{-1}$ assure that the analytic properties of the Hamiltonian are preserved.

Observe, that the bounds on the derivatives of $\varphi$ and $\varphi^{-1}$ imply
\begin{align}
|\varphi(x)|& \geq \frac{1}{\sup\|D\varphi^{-1}\| } \left( |x| - |\varphi^{-1}(0)|\right), \label{phi(x)geq(x)}\\
|\varphi^{-1} (x)| & \leq  |x|\cdot\sup\|D\varphi^{-1}\| + |\varphi^{-1}(0)|.
\end{align}

Since $\varphi$ is a symplectomorphism, in particular it preserves the class of Liouville vector fields.
If $X \in \mathscr{L}(\R^{2n})$ and $\varphi\in \operatorname{Symp}^{ten}(\R^{2n})$, in fact, there exist constants $a,a'>0$, such that the following estimates hold:
\begin{align*}
\| D(D\varphi^{-1} X)(x)\| & \leq \|D^{2}\varphi^{-1}(x)\|\cdot \|X(\varphi^{-1}(x))\|+\|DX(\varphi^{-1}(x))\|\cdot \|D\varphi^{-1}(x)\|^{2}\\
&  \leq  a\|D^{2}\varphi^{-1}(x)\| (|\varphi^{-1}(x)|+1)+a'\\
& \leq  a \|D^{2}\varphi^{-1}(x)\| (|x|\cdot\sup\|D\varphi^{-1}\| + |\varphi^{-1}(0)| +1) + a'<+\infty,
\end{align*}
that is, $D\varphi^{-1} X \in \mathscr{L}(\R^{2n})$.
Since $d(H \circ \varphi)(D\varphi^{-1} X^{\ddagger})(x)=dH(X^{\ddagger})(\varphi(x))>0$ for all $x\in (H\circ \varphi)^{-1}(0)=\varphi^{-1}(H^{-1}(0))$, it follows that $H\circ \varphi$ satisfies condition \ref{h3}.

Similarly, by condition \ref{h1} for $\varphi$, there exist constants $c,c'>0$, such that
$$
d(H \circ \varphi)(D\varphi^{-1} X^{\dagger})(x)  =dH(X^{\dagger})(\varphi(x))\geq c|\varphi(x)|^2 -c',
$$
which combined with \eqref{phi(x)geq(x)} ensures that $H\circ \varphi$ satisfies \ref{h1}.

Observe that whenever $\varphi$ is a symplectomorphism, then $\{H \circ \varphi, F \circ \varphi\}(x)=\{H,F\}(\varphi(x))$. Pre-images of compact sets under diffeomorphisms are compact, in particular, the set
\begin{gather*}
\varphi^{-1}\left(\left\lbrace x\in H^{-1}(0) \ \Big|\  \{H,F\}(x)=0\ \textrm{or}\ \{H,\{H,F\}\}(x)\leq 0 \right\rbrace\right)\\
=\left\lbrace x\in (H\circ\varphi)^{-1}(0) \ \Big|\  \{H,F\}(\varphi( x))=0\ \textrm{or}\ \{H,\{H,F\}\}(\varphi (x))\leq 0 \right\rbrace
\end{gather*}
is compact. Moreover, if $F$ is a coercive function, then $F \circ \varphi$ is also a coercive function. Hence $H\circ \varphi$ satisfies \ref{h4} with $F \circ \varphi$.

 What is left to prove is that $H\circ \varphi$ satisfies \ref{h2}. Since $\varphi\in \operatorname{Symp}^{ten}(\R^{2n})$ and $H$ satisfies \ref{h2}, there exist constants $b, b'>0$, such that for $x\in\R^{2n}$, with $|\varphi(x)|>1$ and $|x|>1$ we have the following estimate:

\begin{align*}
\|D^{3}(H\circ \varphi)(x)\|\cdot |x| & =\left\|D^{3}H(\varphi(x))(D\varphi(x))^{3}+3 D^{2}H(\varphi(x)) D^{2}\varphi(x)D\varphi(x)+DH(\varphi(x))D^{3}\varphi(x)\right\|\cdot |x|\\
& \leq  b \left(\|D^{3}H(\varphi(x))\| \cdot |\varphi(x)|+\|D^{2}\varphi (x)\|\cdot |x|+(|\varphi(x)|+1)\|D^{3}\varphi(x)\|\cdot |x|\right)\\
& \leq b' \left(1+\left( |x| +1\right)\|D^{3}\varphi(x)\|\cdot |x|\right)\\
& \leq b' \left(1+2\|D^{3}\varphi(x)\|\cdot |x|^2\right)  <+\infty.
\end{align*}

The function $\|D^{3}(H\circ \varphi)(x)\|\cdot |x|$ obtains its maximum on the compact set 
$$
\{x\in\R^{2n}\ |\ |\varphi(x)|\leq 1\ \textrm{or} \ |x|\leq 1\},
$$
hence we can conclude that $H\circ \varphi$ satisfies \ref{h2}.
\end{proof}
\section{Symplectic hyperboloids}
\label{sec:SympHyper}
In this last section, we use H{\"o}rmander's symplectic classification of quadratic forms \cite{Hormander1995} to introduce the notion of \emph{symplectic hyperboloids} and to show that a large number of quadratic Hamiltonians are strongly tentacular. 

Suppose $H$ is a non-degenerate quadratic Hamiltonian on $(\R^{2n},\omega_{0})$. Then there exists an affine symplectic change of coordinates such that $H$ can be written in the form
\begin{equation}
H(x):=Q(x,x)-c,\label{quadHam}
\end{equation}
where $Q$ is a non-degenerate quadratic form and $c\in \mathbb{R}$ is a constant.
In particular, for $c>0$ the hypersurface $\Sigma:=H^{-1}(0)$ is diffeomorphic to $S^{k-1}\times \R^{l}$, where $(k,l)$ is the signature of $Q$. In general, a $2n$-dimensional hyperboloid in $\R^{2n}$ can be described as $H^{-1}((-\infty,0))$ for a quadratic Hamiltonian of the form \eqref{quadHam} with $c=1$ and $k,l\geq 1$ (in other words, the quadratic form $Q$ is indefinite). We would like to investigate the equivalence classes of hyperboloids under the action of the linear symplectic group $Sp(\R^{2n})$ and find a unique quadratic form corresponding to each equivalence class under this action.

According to H{\"o}rmander's classification, $\mathbb{R}^{2n}$ splits into a direct sum of subspaces, orthogonal with respect to $Q$ and $\omega$, uniquely determined by the Jordan decomposition of $J_{0}Q$, of one of the following types:
\begin{enumerate}
\item[(a)\namedlabel{a}{(a)}] $S=T^{*}\mathbb{R}^{m}$ and there exists a constant $\lambda>0$ such that 
$$
Q(q,p)=2\lambda\sum_{j=1}^{m}q_{j}p_{j}+2\sum_{j=1}^{m-1}q_{j+1}p_{j}.
$$
Then the Jordan decomposition of $J_0 Q$ has one $m\times m$ box for each of the eigenvalues $\lambda$ and $-\lambda$. The signature of $Q$ is $(m,m)$.
\item[(b)\namedlabel{b}{(b)}] $S=T^{*}\mathbb{R}^{2m}$ and there exist constants $\lambda_{1},\lambda_{2}>0$ such that 
$$
Q(q,p)=2\left(\sum_{j=1}^{2m-2}q_{j}p_{j+2}+\lambda_{1}\sum_{j=1}^{2m}q_{j}p_{j}+\lambda_{2}\sum_{j=1}^{m}(q_{2j}p_{2j-1}-q_{2j-1}p_{2j})\right).
$$
The Jordan decomposition of $J_0 Q$ has one $m\times m$ box for each of the eigenvalues $\pm\lambda_{1}\pm i\lambda_{2}$. The signature of $Q$ is $(2m,2m)$.
\item[(c)\namedlabel{c}{(c)}] $S=T^{*}\mathbb{R}^{m}$ and there exist constants $\mu>0$, $\gamma=\pm 1$ such that 
$$
Q(q,p)=\gamma\left(\mu\sum_{j=1}^{m}q_{j}q_{m+1-j}
-\sum_{j=2}^{m}q_{j}q_{m+2-j}+\mu\sum_{j=1}^{m}p_{j}p_{m+1-j}-\sum_{j=1}^{m-1}p_{j}p_{m-j}\right).
$$
The Jordan decomposition of $J_0 Q$ has one $m\times m$ box for each of the eigenvalues $\pm i\mu$. The signature of $Q$ is $(m,m)$ if $m$ is even and $(m+\gamma,m-\gamma)$ when $m$ is odd.
\end{enumerate}
The number of spaces of each type in this decomposition is uniquely determined.

According to the H\"{o}rmander classification every equivalence class of the action of the linear symplectic group acting on the set of hyperboloids admits a quadratic representant $Q$ in the above form, that is, consisting of blocks of types (a), (b) or (c), which is unique up to a permutation of these blocks. In analogy with the case of positive definite quadratic forms and symplectic ellipsoids, we would like to call \emph{symplectic hyperboloid} a sublevel set of such an indefinite quadratic representant:
\[
W=\{x\in\mathbb{R}^{2n}\,:\, Q(x,x)<1\}.
\]
Our hope is that Rabinowitz Floer homology can eventually be applied to obtain a classification of symplectic hyperboloids, similar to the classification of symplectic ellipsoids obtained as an application of symplectic homology in \cite{FloerHoferWysocki}.

\begin{example}
Consider the Hamiltonian 
\[
H=\frac{1}{2}\left(p_1^2+\ldots +p_n^2+q_1^2+\ldots +q_k^2-q_{k+1}^2-\ldots -q_n^2 \right) -1.
\] 
In this case we get a splitting of $\R^{2n}$ into $n$ two-dimensional subspaces, orthogonal with respect to both the quadratic and the symplectic form:
\[
S_i=\{(p,q)\,:\, p_j=q_j=0\ \textrm{for all}\ j\neq i\}.
\]
These subspaces are of type $\ref{c}$ with $\gamma=1$ for $i=1,\ldots , k$ and they are of type $\ref{a}$ for $i=k+1,\ldots , n$.
\end{example}
\begin{remark}
Note that in the example above the corresponding quadratic representant is not diagonal, whereas the Hamiltonian $H$ obviously is. The criterion for symplectic diagonalisability is presented in \cite{DeLaCruz2014}, where it is shown that a diagonalizable matrix $A$ is symplectically diagonalizable if and only if it satisfies $AJ_0A^TJ_0=J_0A^TJ_0A$.
\end{remark}
\begin{remark}
In \cite{CieliebakEliashbergPolterovich2017}, the authors compute the symplectic homology of symplectic hyperboloids such as the ones in the example above, allowing for different coefficients of the positive and negative definite quadratic parts. These examples always decompose into $2$-dimensional subspaces of type $\ref{a}$ and type $\ref{c}$ (with $\gamma=1$). In particular, they are strongly tentacular, as we will show below.
\end{remark}
The next proposition can be applied in order to determine whether a given quadratic Hamiltonian is strongly tentacular, just on the basis of the orthogonal decomposition and the type of subspaces in H\"{o}rmander's classification. For instance, one immediate consequence will be that the Hamiltonian in the example above is in fact strongly tentacular, for any $1\leq k \leq n$. In turn, this implies that the Rabinowitz Floer Homology can be defined.
 
\begin{proposition}
Consider $(\R^{2n},\omega_{0})$ and a Hamiltonian $H$ defined by a non-degenerate quadratic form $Q$. Decompose $(\R^{2n},\omega_{0})$  into symplectic subspaces $S_i$ which are orthogonal with respect to $\omega_{0}$ and $Q$ and let $Q_i$ denote the restriction of $Q$ to $S_i$. Suppose the pair $(S_i, Q_i)$ belongs to one of the following types in H\"{o}rmander's classification:
\begin{enumerate}[label*=(\alph*)]
\item[(1)\namedlabel{1}{(1)}] type \ref{a} with $m=1$, or $m=2$ and $\lambda>\frac{1}{\sqrt{2}}$, or $m>2$ and $\lambda>2$;
\item[(2)\namedlabel{2}{(2)}] type \ref{b} with $m=1$, or $m=2$ and $\lambda_1>\frac{1}{\sqrt{2}}$, or $m>2$ and $\lambda_{1}>2$;
\item[(3)\namedlabel{3}{(3)}] type \ref{c} with $m=1$ and $\gamma=1$.
\end{enumerate}
Then $H$ is strongly tentacular.
\label{prop:quadTentHam}
\end{proposition}

\begin{proof}
\textbf{Step 1: } We will start by proving that the set of strongly tentacular Hamiltonians is invariant under linear symplectic changes of coordinates.
Suppose $H$ is strongly tentacular and $\phi:\R^{2n}\to\R^{2n}$ is a linear symplectic transformation. Observe that if $X$ is an asymptotically regular Liouville vector field, then $\phi^{*}X$ is too, since $\phi^{*}\omega_0=\omega_0$. As a result, axioms \ref{h1}--\ref{h3} are satisfied for $H\circ \phi$. Analogously, if $F$ is a coercive function, then $F\circ \phi$ is also coercive, since the sublevel sets $(F\circ \phi)^{-1}((-\infty, a])=\phi^{-1}(F^{-1}((-\infty,a])$ are compact for all $a\in \R$. Moreover, $\{ H\circ \phi, F\circ \phi\}(x)= \{H, F\} (\phi(x))$ and $\{H\circ \phi, \{ H\circ \phi, F\circ \phi\}\}(x)= \{H,\{H, F\}\} (\phi(x))$, hence $H\circ \phi$ satisfies also axiom \ref{h4}.

\textbf{Step 2: } It follows from step 1 that, without loss of generality, we can assume $H(x):=\frac{1}{2}Q(x,x)-c$. We prove that if a Hamiltonian of this form satisfies the assumptions of the proposition, then it satisfies hypotheses \ref{h1} and \ref{h3}.

By assumption we can write the Hamiltonian as $H(x)=\frac{1}{2}\sum_i Q_i(x_{i},x_{i})-c$, where $x_i\in S_i$ and $(S_i, Q_i)$ is of type \ref{1}, \ref{2} or \ref{3}. Suppose that for every $i$ the Hamiltonian $H_i(x_i):=\frac{1}{2}Q_i(x_{i},x_{i})-c$ satisfies the hypotheses $dH_i(X^i)(x_i)\geq c_i |x_i|^2$ for $x_i\in S_i$, with $X^i$ an asymptotically regular Liouville vector field and a constant $c_i>0$. Then $X^\dagger:=\sum_i X^i (x_i)$ is an asymptotically regular Liouville vector field on $T^* \R^m$, such that $dH_x(X^\dagger)=\sum dH_i(X^i)(x_i)\geq \min_i c_i |x|^2$. In other words, $H=\sum_i H_i$ satisfies hypotheses \ref{h1} and \ref{h3}.

Therefore it suffices to prove hypotheses \ref{h1} and \ref{h3} for Hamiltonians $H_i$ in each of the three distinct cases \ref{1}, \ref{2} and \ref{3}. Our proof of existence will be in fact a constructive proof. In the standard coordinates on $(T^{*}\mathbb{R}^{m},\omega_{0})$, for each $\alpha \in \R$ we define the vector field
$$
X^{\alpha}(q,p) = \sum_{i=1}^{m}\left(\left(\tfrac{1}{2}p_{i}+\alpha q_{i}\right)\partial_{p_{i}} + \left(\tfrac{1}{2}q_{i}+\alpha p_{i}\right)\partial_{q_{i}}\right).
$$
A straightforward calculation shows that $X^\alpha$ is an asymptotically regular Liouville vector field for every $\alpha\in \R$.

\textbf{Case \ref{1}:} Let $S_i=T^*\R^m$ and $Q_i$ be of class \ref{a} with coefficient $\lambda>0$. Then 
\begin{gather*}
dH_i(X^{\alpha})(q,p) =\lambda\sum_{j=1}^{m}(\alpha q_{j}^{2}+ q_{j}p_{j}+\alpha p_{j}^{2})+\sum_{j=1}^{m-1}(p_{j} q_{j+1}+\alpha p_{j}p_{j+1}+\alpha q_{j}q_{j+1})\geq c_i\sum_{j=1}^{m}(q_{j}^{2}+p_{j}^{2}),\\
\textrm{with} \quad c_i  := \left\lbrace\begin{array}{c c}
 \lambda(\alpha-\frac{1}{2}), & m=1,\\
\alpha\left(\lambda-\frac{1}{2}\right)-\frac{1}{2}\left(\lambda+1\right), & m=2,\\
\alpha\left(\lambda-1\right)-\frac{1}{2}\left(\lambda+1\right), & m>2.
\end{array}\right.
\end{gather*}
The last inequality is obtained by square completion. We can see that for $m=1$, the constant $c_i$ is positive provided $\alpha>\frac{1}{2}$; for $m=2$ the constant $c_i$ is positive, provided $\lambda>\frac{1}{2}$ and $\alpha>\frac{\lambda+1}{2\lambda-1}$; finally, for $m>2$, the constant $c_i$ is positive, provided $\lambda>1$ and $\alpha>\frac{\lambda+1}{2(\lambda-1)}$. Therefore, in all cases, $H_i$ satisfies hypotheses \ref{h1} and \ref{h3} on $S_i$.

\textbf{Case \ref{2}:} Let $S_i=T^*\R^{2m}$ and $Q_i$ be of class \ref{b} with coefficients $\lambda_1,\lambda_2>0$. In this case, if we calculate $dH_i(X^{\alpha})$, we obtain:
\begin{align*}
dH_i(X^{a})(q,p) & = \lambda_{1}\sum_{j=1}^{2m}(\alpha q_{j}^{2}+q_{j}p_{j}+\alpha p_{j}^{2})+\lambda_{2}\sum_{j=1}^{m}(p_{2j-1}q_{2j}+q_{2j-1}p_{2j})\\
& \hspace{2.1cm}+ \sum_{j=1}^{2m-2}(p_{j+2}q_{j}+\alpha(p_{j+2}p_{j}+q_{j}q_{j+2}))\geq c_i\sum_{j=1}^{2m}(q_{j}^{2}+p_{j}^{2}),\\
\textrm{with} \quad c_i &:=  \left\lbrace\begin{array}{c c}
\alpha\lambda_1 - \frac{1}{2}(\lambda_{1}+\lambda_{2}), & m=1,\\
\alpha\left(\lambda_{1}-\frac{1}{2}\right)-\frac{1}{2}(\lambda_{1}+\lambda_{2}+1), & m=2,\\
\alpha(\lambda_{1}-1)-\frac{1}{2}(\lambda_{1}+\lambda_{2}+1), & m>2.
\end{array} \right.
\end{align*}
We can see that for $m=1$, the constant $c_i$ is positive provided $\alpha>\frac{\lambda_{1}+\lambda_{2}}{2 \lambda_1}$; for $m=2$ the constant $c_i$ is positive, provided $\lambda_1>\frac{1}{2}$ and $\alpha>\frac{\lambda_1+\lambda_2+1}{2\lambda_1-1}$; finally, for $m>2$ the constant $c_i$ is positive provided $\lambda_1>1$ and $\alpha>\frac{\lambda_{1}+\lambda_{2}+1}{2 (\lambda_1-1)}$, which in turn shows that $H_i$ satisfies hypotheses \ref{h1} and \ref{h3}.

 \textbf{Case \ref{3}:} Let $S_i=T^*\R$ and $Q_i$ be of class \ref{c} with coefficients $\mu>0$ and $\gamma=1$. Then the associated Hamiltonian $H_i$ is of the form $H_i(q_1,p_1)=\frac{\mu}{2}(p_1^2+q_1^2)$. A straightforward calculation shows that $dH_i(X^0)=\mu (p_1^2+q_1^2)$, which proves the claim in this last case.\\

\textbf{Step 3: } In this last step we will prove that the Hamiltonian $H$ satisfies \ref{h4}.

For every $i$, let $F_i(x_i):=\frac{1}{2}|x_i|^2$ and assume that for some constants $c_i,c'_i>0$ we have
\begin{eqnarray}
\forall\ i\in I\ & \{H_{i},\{H_{i},F_i\}\}(x_i) \geq c_i|x_i|^2
 &\textrm{and} \qquad
|H_i(x_i)|\leq c'_i (|x_i|^2+1),\label{HHF1}\\
\forall\ i\in I' & \{H_{i},\{H_{i},F_i\}\}(x_i) =0
\qquad &\textrm{and} \qquad
H_i(x_i)\geq c_i |x_i|^2.\label{HHF2}
\end{eqnarray}
Now we take $\varepsilon>0$ and with $F(x):= \frac{1}{2}|x|^2$ we calculate
\begin{align*}
\{H,\{H,F\}\}(x)+ \varepsilon H(x) & = \sum_{i\in I\cup I'}\left( \{H_{i},\{H_{i},F_i\}\}(x_i) +\varepsilon H_i(x_i)\right)-\varepsilon c\\
& \geq \sum_{i\in I}\left( c_i -\varepsilon c'_i\right)|x_i|^2 + \varepsilon\sum_{i\in I}c_i |x_i|^2 -\varepsilon c\,.
\end{align*}
In particular, if we take $\varepsilon$ small enough then we get the estimate
$$
\{H,\{H,F\}\}(x)+ \varepsilon H(x)\geq \varepsilon'|x|^2 -\varepsilon c,
$$
which brings us to the conclusion, that the function above is positive outside a compact set. As a result, for $x\in H^{-1}(0)$ the function $\{H,\{H,F\}\}(x)$ is positive outside a compact set, proving hypothesis \ref{h4}. Therefore, to prove hypothesis \ref{h4} for $H$ it suffices to show that each $H_i$ of type \ref{1}, \ref{2} or \ref{3} satisfies either \eqref{HHF1} or \eqref{HHF2}. 

Let us introduce the following notation: Let $J_0$ be the standard almost complex structure on $T^*\R^m$ and let $A$ be a real, symmetric, $2m \times 2m$ matrix, such that $H(x)=\frac{1}{2} Q(x,x)-c=\frac{1}{2} \langle x, A x\rangle -c$. Then for the radial function $F$ the corresponding Poisson brackets are given by
\begin{align*}
\{H, F\}(x) & := \langle x, \tfrac{1}{2}(J_0 A +(J_0 A)^T)x\rangle, \\
\{H,\{H, F\}\}(x) & := \langle x, (A^2+\tfrac{1}{2}((J_0 A)^2 +((J_0 A)^2)^T))x\rangle.
\end{align*}
In particular, $\{H,\{H, F\}\}(x)>0$ for all $x\neq 0$, whenever the matrix $A^2+\tfrac{1}{2}((J_0 A)^2 +((J_0 A)^2)^T)$ is positive definite. We define 
\begin{equation}
\overline{B}:=A^2+\tfrac{1}{2}((J_0 A)^2 +((J_0 A)^2)^T).
\label{ovB}
\end{equation}
Our goal now is to verify in every case when the corresponding matrix $\overline{B}$ is positive definite. Note, that for every orthonormal matrix $O$ the matrices $\overline{B}$ and $O^T\overline{B}O$ have the same eigenvalues. In particular, permutations of the coordinates in which we write $\overline{B}$ correspond to multiplying $\overline{B}$ by an orthonormal matrix $O$ to obtain $O^T\overline{B}O$ and thus it does not change its eigenvalues. Therefore, in order to verify if $\overline{B}$ is positive definite, we can permute the coordinates in the base, even if the coordinate change is not symplectic.

\textbf{Case \ref{1}:} 
Let $S_i=T^*\R^{m}$, let $Q_i$ be of class \ref{a} with coefficient $\lambda>0$, and let $A_i$ be the matrix corresponding to $Q_i$.
Let $\overline{B}_i$ be as in \eqref{ovB}. If we write $\overline{B}_i$ in the coordinate base $q_1,\dots q_m,p_m,\dots p_1$, then $\overline{B}_i$ is a block matrix with two identical $m$ by $m$ matrices $B$ on the diagonal and $0$ elsewhere, where $B=\{b_{l,k}\}_{l,k=1}^{m}$ is defined
\begin{equation}
\label{B}
b_{l,k}:= \left\lbrace\begin{array}{c l}
2\lambda^2, &\quad l=k=1,\\
1+ 2\lambda^2, &\quad l=k= 2, \dots m,\\
2\lambda, &\quad |k-l|=1,\\
\frac{1}{2}, &\quad |k-l|=2,\\
0, &\qquad \textrm{otherwise}.
\end{array}\right.
\end{equation}
Naturally, $\overline{B}_i$ is positive definite if and only if $B$ is positive definite. 

In particular, for $m=1$ the corresponding matrix $\overline{B}_i$ has just one entry equal to $2\lambda^2$, hence it is always positive definite. For $m=2$ the associated matrix $B$ has eigenvalues 
$\frac{1}{2} \left(1+4 \lambda ^2\pm\sqrt{16 \lambda ^2+1}\right)$, thus it is positive definite for $\lambda>\frac{1}{\sqrt{2}}$. Finally, for $m>2$ the corresponding matrix $B$ is diagonally dominant provided $\lambda>2$, hence by a theorem of Levy-–Desplanques (cf. \cite{levy1881} and \cite{desplanques1887}), $B$ is in this case positive definite. We can conclude that for in all of the above cases \eqref{HHF1} is satisfied.

\textbf{Case \ref{2}:} Let $S_i=T^*\R^{2m}$, let $Q_i$ be of class \ref{b} with coefficients $\lambda_1,\lambda_2>0$, and let $A_i$ be the matrix corresponding to $Q_i$.
Let $\overline{B}_i$ be as in \eqref{ovB}. If we write $\overline{B}_i$ in the coordinate base $p_1,\dots p_{2m-1},q_{2m-1},\dots q_1,p_2,\dots p_{2m},q_{2m},\dots q_2$, then in this base the corresponding $\overline{B}_i$ is a block matrix with four identical $m$ by $m$ matrices $B$ on the diagonal and $0$ elsewhere, where $B$ is the matrix from \eqref{B} with parameter $\lambda_1$ instead of $\lambda$.
Naturally, $\overline{B}_i$ is positive definite if and only if $B$ is positive definite. By the arguments presented above $\overline{B}_i$ for $m=1$ is always positive definite, for $m=2$ it is positive definite whenever $\lambda_1>\frac{1}{\sqrt{2}}$ and it is also positive definite for $m>2$ and $\lambda_1>2$.

We can conclude that in all of the above cases \eqref{HHF1} is satisfied.

\textbf{Case \ref{3}:} Let $S_i=T^*\R$ and $Q_i$ be of class \ref{c} with coefficients $\mu>0$ and $\gamma=1$. Then the associated Hamiltonian $H_i$ is of the form $H_i(q_1,p_1)=\frac{\mu}{2}(p_1^2+q_1^2)$, hence it is always positive accept at $0$. On the other hand the associated matrix $\overline{B}_i=0$. We can conclude that in this case \eqref{HHF2} is satisfied.
\end{proof}
\begin{remark}
It follows from Proposition \ref{prop:quadTentHam} that Rabinowitz Floer homology is well-defined for the Hamiltonians corresponding to one of the three types in the statement of the proposition. Moreover, in view of Corollary \ref{cor:SympTentRFH} and Lemma \ref{lem:SympTen}, it will also be well-defined for the composition of such Hamiltonians with any symplectomorphism. Finally, by invariance the homology will also be well-defined for any sufficiently small compactly supported perturbation of any of the previous Hamiltonians. 
\end{remark}
\section*{Acknowledgment}
We would like to thank Will Merry, Bente Bakker and Paul Biran for fruitful discussions on this subject. We would also like to thank Alexander Fauck for sharing with us the preliminary version of his PhD thesis. We would like to thank the anonymous referees for carefully reading our manuscript, and for providing useful and detailed comments on it.

This work has been supported by the NWO Grant 613.001.111 \emph{Periodic motions on non-compact energy surfaces}. The third author has also been supported by the SNF Grant 200021\_182564 \emph{Periodic orbits on non-compact hypersurfaces}.
\bibliographystyle{plain}
\bibliography{RabFloer}

\end{document}